\documentclass[12pt]{amsart}
\usepackage{times,amsfonts,amsmath,amstext,amsbsy,amssymb,amsopn,amsthm,upref,eucal,amscd}
\usepackage[T1]{fontenc}
\usepackage{times,amsfonts,amsmath,amstext,amsbsy,amssymb,
amsopn,amsthm,upref,eucal}\usepackage[T1]{fontenc}
\usepackage{srcltx}

\newcommand \R {{ \mathbb R}}
\newcommand \C {{ \mathbb C}}
\newcommand \Z {{ \mathbb Z}}
\newcommand \Q {{ \mathbb Q}}
\newcommand \re {{ \operatorname{Re} }}
\newcommand \im {{ \operatorname{Im} }}
\newcommand \Cal { \mathcal }
\newcommand{\F}{\mathcal F}
\renewcommand{\H}{\mathcal H}
\renewcommand{\L}{\mathcal L}

\newcommand{\<}{{\langle}}
\renewcommand{\>}{{\rangle}}

\newcommand{\SL}{\operatorname{SL}(2,{\mathbb R})}
\newcommand{\SO}{\operatorname{SO}(2,{\mathbb R})}

\newtheorem {lemma} {Lemma}
\newtheorem {theorem} {Theorem}

\newtheorem{corollary}{Corollary}  
\newtheorem{remark}{Remark}
\newtheorem{definition}{Definition}

\title[Non-uniform hyperbolicity of the Kontsevich--Zorich cocycle]%
{A geometric criterion \\ for the non-uniform hyperbolicity \\ of the Kontsevich--Zorich cocycle}

 \author{Giovanni Forni \\ with an Appendix by Carlos Matheus }

\address{Giovanni Forni \\ Department  of Mathematics\\ University of Maryland \\
  College Park, MD USA}
  \email{gforni@math.umd.edu}

\address{Carlos Matheus\\Coll\`ege de France\\  3, Rue d'Ulm\\ Paris, CEDEX 05, France}
\email{matheus@impa.br.}

\begin{document}
 
  \begin{abstract}
 We prove a geometric criterion on a $\SL$-invariant ergodic probability measure 
on the moduli space of holomorphic abelian differentials on Riemann surfaces 
for the non-uniform hyperbolicity of the Kontsevich--Zorich cocycle on the real Hodge
bundle. Applications include measures supported on the $\SL$-orbits of all algebraically 
primitive Veech surfaces (see also \cite{Bouw:Moeller}) and of all Prym eigenforms 
discovered in \cite{McMullen2}, as well as all canonical absolutely continuous 
measures on connected components of strata of the moduli space of abelian differentials 
(see also \cite{Ftwo}, \cite{Avila:Viana}). The argument simplifies and generalizes
our proof for the case of canonical measures \cite{Ftwo}. In an Appendix Carlos Matheus
discusses several relevant examples which further illustrate the power  and the limitations
of our criterion.
 \end{abstract}

 \maketitle
 
\section{Introduction} The non-uniform hyperbolicity of the Kontsevich--Zorich
cocycle with respect to the canonical absolutely continuous $\SL$-invariant probability
measures on the connected components of the moduli space of abelian differentials 
was originally conjectured by Zorich and Kontsevich. The conjecture was later proved by 
the author in \cite{Ftwo}. In this paper we develop the ideas of \cite{Ftwo} and prove a geometric 
criterion for the non-uniform hyperbolicity of the Kontsevich--Zorich cocycle with respect to a general 
$\SL$-invariant probability ergodic measure. This criterion yields a streamlined proof of the 
original argument in \cite{Ftwo} for the case of the canonical measures. In fact, this paper is an 
outgrowth of an unpublished note written to clarify and rectify parts of the original proof of Theorem
8.5 in \cite{Ftwo} (see also \cite{Krikorian}) in response to questions of the participants of the  \emph {Groupe de Travail in Geometry and Dynamics }of the Universit\'e de Paris VI during the academic year 2003/2004. We are extremely grateful to all the participants, in particular to Artur Avila,  
S\'ebastien Gou\"ezel and Rapha\"el Krikorian for their comments and questions. We are also very grateful to A.~Bufetov who encouraged us to formulate and write up the criterion we present below. 
Last but not least, we thank A.~Zorich and C.~Matheus for many fruitful discussions about the Kontsevich--Zorich spectrum which lead in particular to our joint work \cite{Forni:Matheus:Zorich:one},  \cite{Forni:Matheus:Zorich:two} and to the Appendix by C.~Matheus at the end of this paper.

The non-uniform hyperboliciy of the Kontsevich--Zorich cocycle is important in several
applications to the dynamics of translation flows and interval exchange transformations,
such as the fine behavior of ergodic averages and their deviations~\cite{Zorich1}-\cite{Zorich5}, \cite{Kontsevich}, \cite{Ftwo}, \cite{Bufetov1}, \cite{Bufetov2}. The non-vanishing of the second exponent Kontsevich--Zorich exponent is crucial in the approach of A.~Avila and the author to weak mixing for translation flows and interval exchange transformations \cite{Avila:Forni}. While the Lebesgue generic case is now well-understood, there are only a few scattered results on the Kontsevich--Zorich spectrum for generic translation flows with respect to $\SL$-invariant probability ergodic measures other than the canonical ones. Most of these results, for instance \cite{Bouw:Moeller} and  \cite{Forni:Matheus:Zorich:one} (as well as work in progress by C.~Matheus, M.~M\"oller and J.-C.~Yoccoz) are limited to $\SL$-invariant probability measures associated to Veech surfaces and leave open natural questions such as the non-uniform hyperbolicity of the Kontsevich--Zorich cocycle for strata of quadratic differentials. 
Finally, we should stress that, as far as we can see, our methods cannot yield the stronger property of \emph{simplicity} of the Kontsevich--Zorich spectrum, which was proved by A.~Avila and M.~Viana \cite{Avila:Viana} for the canonical absolutely continuous measures on the moduli space of abelian differentials. It would be interesting, in our opinion, to formulate geometric criteria for the simplicity of the spectrum based on Avila-Viana's approach. However, since there exist  non-uniformly hyperbolic $\SL$-invariant ergodic measures for which the spectrum is not simple, the criterion we propose in this paper has a somewhat different scope (while of course yielding a weaker property). For instance, for a particular sequence of square-tiled cyclic covers, the so-called ``stairs''  square-tiled cyclic covers (double covers of square-tiled ''stairs'' surfaces), all the Kontsevich--Zorich exponents of the associated 
$\SL$-invariant measure are non-zero and all but the top and bottom exponents are double (explicit values for all the exponents are computed in \cite{Eskin:Kontsevich:Zorich:cyclic}, Appendix B.1). In fact, the non-uniform hyperbolicity of the Kontsevich--Zorich cocycle for all ``stairs''  square-tiled surfaces  follows easily from our criterion (see \cite{Eskin:Kontsevich:Zorich:cyclic}, Appendix C, for the verification of the main condition of our criterion, the \emph{Lagrangian }property of the horizontal/vertical foliation introduced in \cite{Ftwo}, Definition 4.3, and recalled below in Definition \ref{def:LagrangianFol}). 
Other examples of  measures on strata of abelian differentials on surfaces of genus $3$ with multiple Kontsevich--Zorich exponents to which our criterion applies have been found by C.~Matheus. 
He communicated to the author that he can prove the following result: $(a)$ all measures on 
$\H(1,1,1,1)$ coming from a regular (unbranched) double cover construction over a square-tiled 
surface in the stratum $\H(1,1)$ and $(b)$ the measure on $\H(2,2)$ coming from a 
certain regular double cover construction over an $L$-shaped square-tiled surface with three
squares (in the stratum $\H(2))$ have a double second exponent. The details and the complete
argument for case $(b)$ are given in the Appendix  \ref{s.2covers} by C.~Matheus at the end of this
paper.

In genus $2$, the simplicity of the Kontsevich--Zorich spectrum for the canonical absolutely
continuous $\SL$-invariant probability measures on the space of abelian differentials was 
proved in \cite{Ftwo}. In unpublished work the author was able to extend this result
to all $\SL$-invariant probability ergodic measures. Later, M.~Bainbridge \cite{Bainbridge}
was able to establish the exact numerical values of the second exponent (the only non-trivial one) 
for all such measures. Bainbridge result, based on C.~McMullen's classification theorem 
for $\SL$-invariant probability measures in genus $2$ \cite{McMullen1}, implies in particular the simplicity of the Kontsevich--Zorich spectrum. 

\medskip
In order to state our geometric criterion we recall a few fundamental properties 
of probability measures invariant under the Teichm\"uller flow and introduce a couple 
of definitions. Let $\H_g$ be the moduli space of abelian differentials of unit
total area on Riemann surfaces of genus $g\geq 2$. The space $\H_g$ has
a natural stratification defined as follows. For any $\kappa:=(k_1, \dots, k_\sigma) 
\in \Z_+^\sigma$ such that $\sum k_i =2g-2$, the subset $\H(\kappa)\subset
\H_g$ of all abelian differentials which have exactly $\sigma \in \Z_+$ zeroes
$p_1, \dots, p_\sigma$ with multiplicities $k_1, \dots, k_\sigma$ is non-empty.
Each stratum $\H(\kappa)$ has the structure of an affine orbifold 
locally parametrized by the period map into the relative complex cohomology 
$H^1(S, \Sigma;\C)$ of the surface $S$ relative to zero set $\Sigma:=\{p_1, \dots, p_\sigma\}$.
The strata are in general not connected and can have up to three connected 
components \cite{Kontsevich:Zorich}. 

\smallskip
Every abelian differential $\omega\in \H_g$ induces a pair $(\F^h_\omega, 
\F^v_{\omega})$ of transverse orientable measured foliations on the topological
surface $S$ of genus $g\geq 2$, called respectively the  \emph{horizontal } and the 
\emph{vertical } foliation of the abelian differential. Such foliations are respectively defined as follows:
\begin{equation}
\F^h_\omega := \{\im \,\omega=0\} \quad \text{ and } \quad \F^v_\omega := \{ \re \,\omega=0\}\,.
\end{equation}
The transverse measure for the horizontal foliation $\F^h_\omega$ is defined by integration
along transverse arcs of the density $\vert \re \,\omega \vert$; the transverse measure for the vertical foliation $\F^v_\omega$ is defined by integration along transverse arcs of the density 
$\vert \im\, \omega \vert$.

We recall that the canonical action of the group $\SL$ on $\H_g$ can be defined as follows.
For any $A\in \SL$ and for any holomorphic differential $\omega$ on a Riemann surface
$S_\omega$, there exists a unique Riemann surface $S_{A\omega}$ and a unique abelian
differential $A\omega$ holomorphic on $S_{A\omega}$ such that 
$$
A\begin{pmatrix}  \re \,\omega \\ \im\, \omega \end{pmatrix} = 
\begin{pmatrix}  \re \,A\omega \\ \im \,A \omega \end{pmatrix} \,.
$$
The \emph{Teichm\"uller geodesic flow} is defined as the action of the diagonal subgroup
$\{g_t \}= \{\text{\rm diag} (e^t, e^{-t}) \} < SL(2, \R)$.  It is immediate by the definition that
strata are invariant under the action of $\SL$ on $\H_g$, hence under the
Teichm\"uller flow. In addition, it is possible to define a natural smooth $\SL$-invariant
measure on every stratum. A fundamental result by H.~Masur \cite{Masur2} and W.~Veech 
\cite{Veech1}, \cite{Veech2} states that every such measure has finite mass and is ergodic
 for the Teichm\"uller flow.

\smallskip
The Teichm\"uller flow admits two invariant foliations, $\Cal W^+$ and
$\Cal W^-$ on $\H_g$, which are (locally) defined as follows. The leaves $\Cal W^\pm(\omega)$ 
through any abelian differential $\omega\in \H_g$, are given by the formulas
\begin{equation}
\begin{aligned}
\Cal W^+(\omega) &:= \{ \tilde \omega \in \H_g \vert  \im\, \tilde \omega \in \R^+\cdot  \im\, \omega \} 
=  \{ \tilde \omega \in \H_g \vert  [\F^h_{\tilde \omega}] = [\F^h_\omega] \} \,,\\
\Cal W^-(\omega)& := \{ \tilde \omega \in \H_g \vert  \re\, \tilde \omega \in \R^+\cdot  \re\, \omega \} 
= \{ \tilde \omega \in \H_g \vert  [\F^v_{\tilde \omega}] = [\F^v_\omega] \} \,.
\end{aligned}
\end{equation}
Let $\Cal W^\pm_\kappa$ denote the intersections of the foliations $\Cal W^\pm$ with a stratum $\H(\kappa)$ of the moduli space. For any $\omega \in \H(\kappa)$, the intersection $\Cal W^+_\kappa
(\omega)\cap \Cal W^-_\kappa(\omega)$ coincides with the Teichm\"uller orbit $\{g_t \omega \vert t\in \R\} \subset \H(\kappa)$. For any pair $\omega^+$, $\omega^-\in \H(\kappa)$,  the intersection
$\Cal W^+_\kappa(\omega^+)\cap \Cal W^-_\kappa(\omega^-)$ is non-empty if and only if
the horizontal foliation $\F^h_{\omega^+}$ of $\omega^+$ and the vertical foliation $\F^v_{\omega^-}$ 
of $\omega^-$ are transverse (in the sense of measured foliations). In fact, a pair $(\F^h, \F^v)$
of measured foliations is transverse if and only there exists a holomorphic  abelian differential 
$\omega\in \H_g$ such that $\F^h(\omega) \in \R^+ \cdot \F^h$ and $\F^v(\omega) \in \R^+ \cdot \F^v$.

It is an immediate consequence of a result of \cite{Ftwo} (see Theorem \ref{thm:Tflowhyp} below) that every probability Teichm\"uller-invariant ergodic measure is non-uniformly hyperbolic, in the sense that the tangent cocycle of the Teichm\"uller flow restricted to the relevant stratum has non-zero Lyapunov exponents in all directions except for the flow direction. In fact, this result is a consequence of a stronger property: there exists on every stratum $\H(\kappa)\subset \H_g$ a Hodge Riemannian metric whose restrictions to the leaves of the foliations $\Cal W^+_\kappa$/ $\Cal W^-_\kappa$ is contracted by the backward/forward action of the Teichm\"uller flow (in all directions except for the flow direction), \emph{uniformly} on every compact subset of the (non-compact) space $\H(\kappa)$ (see \cite{Ftwo} and \cite{Athreya:Forni}). The Hodge Riemannian metrics have been explicitly constructed and studied in 
depth in \cite{Athreya:Bufetov:Eskin:Mirzakhani}.

\smallskip
We recall the definition of the Kontsevich--Zorich cocycle over the Teichm\"uller flow \cite{Kontsevich}, \cite{Ftwo}, a continuous-time version of the Rauzy--Veech--Zorich cocyle \cite{Rauzy}, \cite{Veech1}, \cite{Zorich2} over the Rauzy--Veech--Zorich map.

Let $\hat \H_g$ be the Teichm\"uller space of holomorphic abelian differentials of unit total area on Riemann surfaces of genus $g\geq 2$. It can be defined as the moduli space of abelian 
differentials on \emph{marked} Riemann surfaces.  Let $S$ denote the underlying smooth
surface of genus $g\geq 2$. Points of the the Teichm\"uller space $\hat \H_g$ are equivalence 
classes of holomorphic abelian differentials on the surface $S$, endowed with some holomorphic 
structure, with respect to the equivalence relation given by the natural action of the group 
$\text{Diff}^+_0(S)$ of orientation preserving diffeomorphisms isotopic to the identity. 
The moduli space $\H_g$ can be defined as the quotient $\hat \H_g/\Gamma_g$ of the 
Teichm\"uller space of holomorphic abelian differentials of unit total area with respect 
to the action of the \emph{mapping class group} $\Gamma_g:= \text{Diff}^+(S)/\text{Diff}^+_0(S)$ 
on $\hat \H_g$. The Teichm\"uller flow $\{g_t\}$ on the moduli space $\H_g$ lifts to a flow 
$\{\hat g_t\}$ on the Teichm\"uller space $\hat \H_g$.  Let $\{\hat \rho_t \}$ be the
trivial cocycle over the flow $\{\hat g_t\}$ on the trivial cohomology bundle 
$\hat \H_g \times H^1(S,\R)$ defined as follows:
$$
\hat \rho_t :=   \hat g_t \times \text{id} :  \hat \H_g \times H^1(S,\R) \to \hat \H_g \times H^1(S,\R)\,.
$$
The mapping class group $\Gamma_g$ acts on the trivial bundle $ \hat \H_g \times H^1(S,\R)$
by pull-back on each coordinate. The quotient bundle
\begin{equation}
\label{eq:Hodge}
H^1_g :=  \left(  \hat \H_g \times H^1(S,\R) \right) / \Gamma_g 
\end{equation}
is an orbifold vector bundle over the moduli space $\H_g$ of holomorphic abelian differentials
of unit total area called the (real) \emph{Hodge bundle}.  The \emph{Kontsevich--Zorich cocycle}
can be defined as the projection $\{\rho_t\}$ to the Hodge bundle $H^1_g$  of the trivial cocycle
$\{\hat \rho_t\}$. By definition, it is a cocycle over the Teichm\"uller geodesic flow $\{g_t\}$ on the moduli space $\H_g$ of holomorphic abelian differentials of unit total area. 

It is an immediate consequence of the definition that the top exponent of the Kontsevich--Zorich
cocycle is equal to $1$. In addition, since the action of the group $\text{Diff}^+(S)$ on $H^1(S, \R)$ 
is symplectic with respect to the standard symplectic structure given by the intersection form, the cocyle
is \emph{symplectic}, hence for any probability measure $\mu$ on the moduli space $\H_g$, 
invariant under the Teichm\"uller flow and ergodic, its \emph{Lyapunov spectrum} has $g$ 
non-negative and $g$ non-positive exponents (counting multiplicities) and it is symmetric 
with respect to the origin, that is, it has the form
\begin{equation}
\label{eq:KZspectrum}
\lambda_1^\mu=1 \geq \lambda_2^\mu \geq \dots \geq \lambda^\mu_g \geq 
-\lambda_g^\mu \geq \dots \geq -\lambda_2^\mu \geq  -\lambda_1^\mu=-1 \,.
\end{equation}

There is  a simple well-known relation between the Lyapunov spectrum of the Kontsevich--Zorich cocycle and the Lyapunov spectrum of the Teichm\"uller flow (that is, the Lyapunov spectrum of the
tangent cocycle) restricted to any stratum of the moduli space \cite{Zorich2}, \cite{Zorich5}, 
\cite{Kontsevich}, \cite{Ftwo}, \cite{ForniSurvey}. The Lyapunov spectrum of the Teichm\"uller flow with respect to any invariant, ergodic probability measure $\mu$ supported on a stratum $\H(\kappa) \subset \H_g$ of the moduli space can be written as follows in terms of the Lyapunov spectrum of the Kontsevich--Zorich cocycle:
\begin{equation}
\label{eq:Tflowexp}
\begin{aligned}
 2 & \geq (1+\lambda^{\mu}_2)\geq \cdots\geq (1+\lambda^{\mu}_g) \geq 
\overbrace{1=\cdots= 1}^{\sigma-1}  \geq  (1-\lambda^{\mu}_g) \geq \\
&\geq \cdots\geq (1-\lambda^{\mu}_2)\geq 0 \geq -(1-\lambda^{\mu}_2)\geq \cdots 
\geq -(1-\lambda^{\mu}_g) \geq \\
&\geq\underbrace{-1=\cdots= -1}_{\sigma-1}\geq -(1+\lambda^{\mu}_g)
\geq \cdots \geq -(1+\lambda^{\mu}_2)\geq  -2 \,\,. 
\end{aligned}
\end{equation}
It is immediate from the above formula that the non-uniform hyperbolicity of the Teichm\"uller
flow with respect to any ergodic probability measure $\mu$ on $\H(\kappa)$ is equivalent
to a \emph{spectral gap} property for the Kontsevich--Zorich cocycle, namely the strict
inequality $\lambda^\mu_2 < \lambda^\mu_1=1$. For a class of measures 
satisfying a certain integrability condition (class which includes all the canonical absolutely 
continuous measures), the non-uniform hyperbolicity of the Teichm\"uller  flow was proved by
W.~Veech in \cite{Veech2}. In \cite{Ftwo}, Corollary 2.2 (see also \cite{ForniSurvey}, Theorem 5.1), the author proved the following generalization of Veech's result:

\begin{theorem} 
\label{thm:Tflowhyp}
For any Teichm\"uller-invariant ergodic probability measure $\mu$ on
$\H_g$ the Kontsevich--Zorich cocycle has a spectral gap, that is,
$$
\lambda^\mu_2 < \lambda^\mu_1 =1\,,
$$
or, equivalently, any Teichm\"uller-invariant ergodic probability measure  is non-uniformly 
hyperbolic for the Teichm\"uller flow. 
\end{theorem}

The above theorem is a rather straightforward application of the variational formulas for the
Hodge norm on the Hodge bundle (see \cite{Ftwo}, \S 2).

The non-vanishing of the Kontsevich--Zorich exponents is much harder to establish.
For the canonical absolutely continuous $\SL$-invariant measures on strata of the moduli space, 
M.~Kontsevich and A.~Zorich conjectured that the spectrum of the Rauzy-Veech-Zorich cocycle 
(or, equivalently, of the Kontsevich--Zorich cocycle) is \emph{simple}, hence in particular 
the cocycle is \emph{non-uniformly hyperbolic} (that is, all the Lyapunov exponents are 
non-zero). We recall that the non-uniform hyperbolicity was proved by the 
author in \cite{Ftwo} and  the full conjecture was later proved by A.~Avila and M.~Viana 
\cite{Avila:Viana} by a different approach. The non-uniformly hyperbolicity of the Kontsevich--Zorich
cocycle fails in general even for $\SL$-invariant probability measures. The first example of an $\SL$-invariant ergodic probability measure with zero Lyapunov exponents was found by the author in 2002
(later published in  \cite{ForniSurvey}, \S 7). It is the unique $\SL$-invariant probability measure supported on the closed $\SL$-orbit of an `exceptionally symmetric' translation surface of genus $3$, often called the \emph{Eierlegende Wollmilchsau}  surface \cite{Herrlich:Schmithuesen}.  For such a measure the Kontsevich--Zorich spectrum is maximally degenerate (all non-trivial exponents are zero) \cite{ForniSurvey}, \cite{Forni:Matheus:Zorich:one}. Another maximally degenerate example of genus $4$ was later discovered by C.~Matheus and the author \cite{Forni:Matheus}. Both these examples belong 
to the class of  \emph{square-tiled cyclic covers} which includes many examples of partially degenerate
 Lyapunov spectrum  \cite{Forni:Matheus:Zorich:one}, \cite{Eskin:Kontsevich:Zorich:cyclic}.
 Outside of this class, it seems that there are no explicit examples in the literature of $\SL$-invariant probability measures with some zero exponents. However, the latter should not be hard to construct.
 In fact, as pointed out by C.~Matheus, examples should be contained implicitly in the work  
 of C.~McMullen \cite{McMullen3}, although Kontsevich--Zorich exponents are never mentioned there.
 
 For the maximally degenerate examples, the action on homology of the affine group is given by finite symmetry groups \cite{Matheus:Yoccoz}. According to a theorem of M\"oller \cite{Moeller2}, there are 
no other maximally degenerate examples coming from Veech surfaces (except possibly in genus $5$). For square-tiled cyclic cover, this result can be derived from the Kontsevich--Zorich formula (see
\cite{Forni:Matheus:Zorich:one}).  Recently A.~Avila and M.~M\"oller have announced work that confirms the conjectured negative answer to the question on the existence of $\SL$-invariant measures with completely degenerate Kontsevich--Zorich spectrum  outside of the class of Veech surfaces.

\medskip
Our criterion for the non-uniform hyperbolicity of the Kontsevich--Zorich cocycle 
applies to certain $\SL$-invariant probability ergodic measures which have a local 
product structure in the sense defined below.
 
For every open subset $\Cal U \subset \H(\kappa)$, the local invariant foliations $\Cal W^\pm_{\Cal U}$
are defined as follows:  the leaf $\Cal W^\pm_{\Cal U}(\omega)$ is the unique connected component of the intersection $\Cal W^\pm_\kappa (\omega) \cap \Cal U$ which contains the abelian
differential $\omega\in \Cal U$.  

 \begin{definition} 
 \label{def:prodset}
 An open set $\Cal U \subset \H(\kappa)$ is said to be of \emph{product type}
 if, for any pair of abelian differentials $(\omega^+, \omega^-) \in \Cal U \times \Cal U$, there exist an abelian differential $\omega \in \Cal U$ and an open interval $(a,b)\subset \R$ such that
 $$
 \Cal W^+_{\Cal U}(\omega^+) \cap  \Cal W^-_{\Cal U}(\omega^-) =\{ g_t \omega \vert t\in (a,b)\}\,.
 $$
 \end{definition}
 Since every stratum $\H(\kappa)$ of the moduli space of abelian differentials has an affine structure with local charts given by the relative period map, by writing the invariant foliations $\F^\pm_\kappa$ in coordinates, it can be verified that the topology $\H(\kappa)$ has a (countable) basis of open sets of product type.
 
 Let $\Cal U\subset \H(\kappa)$ be an open set of product type. For every set $\Omega 
 \subset  \Cal U$, let
$$
\Cal W^\pm_{\Cal U}(\Omega) := \bigcup_{\omega\in \Omega} \Cal W^\pm_{\Cal U}(\omega)\,.
$$
 \begin{definition} 
 \label{def:prodstruct}
A Teichm\"uller-invariant measure $\mu$ supported on $\H(\kappa)$ is said to have a \emph{product structure }on an open subset $\Cal U \subset \H(\kappa)$ of product type if, for any pair of Borel subsets $\Omega^+$, $\Omega^- \subset \Cal U$, 
 $$
 \mu(\Omega^+)\not =0 \text{ and }  \mu(\Omega^-)\not =0 \Rightarrow   
 \mu\left( \Cal W^+_{\Cal U}(\Omega^+) \cap \Cal W^-_{\Cal U}(\Omega^-)\right) \not = 0\,.
 $$
 A Teichm\"uller-invariant measure $\mu$ on $\H(\kappa)$ is said to have a \emph{local product structure} if every abelian differential $\omega \in \H(\kappa)$ has an open neighborhood 
 $\Cal U_\omega \subset  \H(\kappa)$ of product type on which $\mu$ has a
 product structure.
 \end{definition}
 
 We remark that any $\SL$-invariant \emph{affine }measure has a local product structure.
 A measure on a stratum of the moduli space of abelian differentials is called affine
if it is locally equal (up to normalization) to the the restriction of the Lebesgue measure to a 
complex affine subspace with respect to the natural affine structure induced by the complex 
(relative) period map. It is part of a broader Ratner-type conjecture on the action of $\SL$ on the 
moduli space of abelian differentials that all  $\SL$-invariant probability measures are affine.  
A proof  of a Ratner-type conjecture, which includes the statement that every $\SL$-invariant
probability measure is affine, has been recently announced by A.~Eskin and M.~Mirzakhani 
\cite{Eskin:Mirzakhani}. For the genus two case, all $\SL$-invariant probability measures 
were classified by C.~McMullen \cite{McMullen1} and K.~Calta \cite{Calta} and are known 
to be affine.

\smallskip
Next we recall the notion of a completely periodic \emph{Lagrangian }measured foliation (introduced
in  \cite{Ftwo}, Definition 4.3):
 \begin{definition} 
 \label{def:LagrangianFol}
 A \emph{completely periodic }measured foliation ${\F}$ on a compact orientable surface $S$ of genus $g\geq 2$ is a measured foliation on $S$ such that all its regular leaves are closed (compact) curves.  
 
 The  \emph{homological dimension }of a completely periodic measured foliation $\F$
 on $S$ is the dimension of the (isotropic) subspace ${\L}({\F}) \subset H_1(S,{\R})$, generated by the homology classes of the regular leaves of ${\F}$. 

A \emph{Lagrangian }measured foliation ${\F}$ on $S$ is a completely periodic measured foliation
of maximal homological dimension (equal to the genus of the surface), that is, a measured foliation
such that  the subspace ${\L}({\F})$  is a Lagrangian subspace of the space $H_1(S,\R)$, endowed with the symplectic structure given by the intersection form. 
\end{definition}

A completely periodic measured foliation ${\F}$ is Lagrangian if and only if it has $g\geq 2$ distinct regular leaves $\gamma_1,\dots,\gamma_g$ such that ${\widehat S}:=S\setminus\cup \{\gamma_1,\dots,\gamma_g\}$ is homeomorphic to a sphere minus $2g$ (paired) disjoint disks.

\begin{definition} 
\label{def:cuspLagragian}
A Teichm\"uller-invariant probability measure on a stratum $\H(\kappa)$  is called \emph{cuspidal }if it has a local product structure and its support contains a holomorphic differential with completely periodic horizontal or vertical foliation. 

The \emph{homological dimension }of a Teichm\"uller-invariant 
measure  is the maximal homological dimension of  a completely periodic vertical or horizontal 
foliation of a holomorphic differential in its support.

A Teichm\"uller-invariant probability measure is said to be \emph{Lagrangian} if it has maximal homological dimension, that is, if its support contains a holomorphic differential with Lagrangian horizontal or vertical foliation.
\end{definition}

We are very grateful to J.~Smillie who pointed out to us that any closed $\SL$-invariant set, hence in
particular the support of any $\SL$-invariant measure, contains a holomorphic differential with completely periodic horizontal or vertical foliation. In fact, according to a theorem by Smillie and B.~Weiss \cite{Smillie:Weiss}, any closed $\SL$-invariant subset of the moduli space contains a 
minimal set for the Teichm\"uller horocycle flow and every such minimal set corresponds to a cylinder decomposition.  By this result and by the Ratner's type result announced by A.~Eskin and 
M.~Mirzakhani \cite{Eskin:Mirzakhani} that every $\SL$-invariant probability measure is affine, 
it follows that every $\SL$-invariant probability measure is cuspidal.

\smallskip
Our main result is the following criterion:

\begin{theorem} 
\label{thm:main}
Let $\mu$ be a $\SL$-invariant ergodic probability measure on a
stratum $\H(\kappa)$ of the moduli space of abelian differential.
If $\mu$ is cuspidal Lagrangian, the Kontsevich--Zorich cocycle is non-uniformly
hyperbolic $\mu$-almost everywhere. In fact, the Lyapunov exponents $\lambda^\mu_1 
\geq \dots \geq \lambda^\mu_{2g}$ of the Kontsevich--Zorich cocycle form a symmetric subset 
of the real line and the following inequalities hold:
\begin{equation}
\label{eq:NUH}
\begin{aligned}
 1=\lambda^\mu_1>\lambda^\mu_2\geq& \dots \geq\lambda^\mu_g>0>\lambda^\mu_{g+1}
 =-\lambda^\mu_g\geq
\dots \\
              &\dots \geq \lambda^\mu_{2g-1}=-\lambda^\mu_2 > \lambda^\mu_{2g}=-\lambda^\mu_1=-1 \,\,.
 \end{aligned}
 \end{equation}
\end{theorem}

A weaker statement can be proved without assuming that the $\SL$-invariant probability measure has
a local product structure. In fact, the following holds:

\begin{theorem} 
\label{thm:partial}
Let $\mu$ be a $\SL$-invariant ergodic probability measure on a stratum $\H(\kappa)
\subset \H_g$ of the moduli space of abelian differential on Riemann surfaces of genus
$g\geq 3$. If $\mu$ is Lagrangian, the following inequalities hold for the Lyapunov spectrum of 
the Kontsevich--Zorich cocycle:
\begin{equation}
\label{eq:partial}
1=\lambda^\mu_1> \lambda^\mu_2 \geq \dots \geq \lambda^\mu_{[g+1/2]}>  0 \,.
\end{equation}
\end{theorem}
As recalled above, in genus $2$ the Kontsevich--Zorich spectrum is simple for all $\SL$-invariant
ergodic probability measures \cite{Bainbridge}. This result can be derived from the
Kontsevich--Zorich formula (see Corollary~\ref{cor:KZformula} below) and from a description of
possible boundary points of Teichm\"uller disks in the Deligne-Mumford compactification of the
moduli space.

For cuspidal measures, Theorem~\ref{thm:main} can be generalized as follows:

\begin{theorem} 
\label{thm:general}
Let $\mu$ be a $\SL$-invariant ergodic probability cuspidal measure on a
stratum $\H(\kappa)$ of the moduli space of abelian differential.
If $\mu$ has homological dimension $k\in \{1, \dots, g\}$, there are at least
$k$ strictly positive Kontsevich--Zorich exponents, that is,
$$
1=\lambda^\mu_1>\lambda^\mu_2\geq \dots \geq\lambda^\mu_k>0\,.
$$
\end{theorem}
The proof of the above theorem can be obtained along the same lines of the proof of
Theorem~\ref{thm:main} and will not be explained in detail in this paper. 

We remark that the lower bound given in Theorem~\ref{thm:general} is optimal in general. 
In fact, the \emph{maximally degenerate examples }of \cite{ForniSurvey}, \cite{Forni:Matheus} (see 
also \cite{Forni:Matheus:Zorich:one}) are both given by $\SL$ invariant probability measures (supported 
on closed $\SL$-orbits) of \emph{homological dimension equal to} $1$ (in both cases completely
periodic directions split the surface into two homologous cylinders). However, there exist  examples 
of cuspidal measures where the number of strictly positive Kontsevich--Zorich exponents is \emph{greater} than the homological dimension of the measure. For instance the family of square-tiled cyclic covers studied by C.~Matheus and J.-C.~Yoccoz (see  \cite{Matheus:Yoccoz}, \S 3.1) provide examples of $\SL$-invariant probability measures on the strata $\H(q-1, q-1, q-1)$ for any odd $q\geq 3$  with homological dimension equal to $1$ and a number of non-zero Kontsevich--Zorich exponents equal
to  $1 + (q-3)/2$ when $q=3 (\text{\rm mod. }4)$, and $1+ (q-1)/2$ when $q=1 (\text{\rm mod. }4)$.
We are grateful to C.~Matheus who computed the above formulas for us. His calculations appear
at the end of this paper in Appendix \ref{s.jc}. 

The case $q=3$ in the Matheus--Yoccoz  family corresponds to the example by C.~Matheus and the author of a square-tiled cyclic cover of genus $4$ with completely degenerate spectrum \cite{Forni:Matheus}, \cite{Forni:Matheus:Zorich:one} mentioned above.

Recently, V. Delecroix and C.~Matheus have found examples of cuspidal non-uniformly hyperbolic 
$\SL$-invariant probability measures which are not Lagrangian: the measures supported on the $\SL$ orbits of a couple of square-tiled surfaces, one in genus $3$ and one in genus $4$. Such examples were found with the aid of A. Zorich's computer program to compute the exponents to have simple
Lyapunov spectrum and, as communicated to us by  Carlos Matheus,  it seems possible to prove simplicity by some version of Avila--Viana's criterion \cite{Avila:Viana}.


As we have remarked above, by a result of J.~Smillie and B.~Weiss  \cite{Smillie:Weiss} on
minimal sets for the Teichm\"uller horocycle flow and by the Ratner type result announced by 
A.~Eskin and M.~Mirzakhani \cite{Eskin:Mirzakhani} it follows that every $\SL$-invariant probability measure is cuspidal. However, not all $\SL$-invariant probability ergodic measures are Lagrangian. In fact, as mentioned above, there are many examples among $\SL$-invariant measures supported on 
$\SL$-orbits of square-tiled cyclic covers  which fail to be non-uniformly hyperbolic \cite{Forni:Matheus:Zorich:one}, \cite{Eskin:Kontsevich:Zorich:cyclic}. By our criterion such measures are 
not Lagrangian.  To the author's best knowledge, no examples of (cuspidal)  $\SL$-invariant probability measures which are not Lagrangian outside of the class of measures supported on closed $\SL$-orbits or on strata of branched covers are known.

All the $\SL$-invariant probability measures supported on $\SL$-orbits of \emph{algebraically primitive} Veech surfaces are cuspidal Lagrangian (Lemma~\ref{lemma:bouillabaisse}). We are very grateful to 
P.~Hubert for explaining to us the proof of this basic fact. Our criterion therefore implies a corollary of formulas by I.~Bouw and M.~M\"oller \cite{Bouw:Moeller} which states that  the Kontsevich--Zorich cocycle is non-uniformly hyperbolic for all algebraically primitive Veech surfaces (Corollary 
\ref{cor:KZexpalgrank}). In fact, in section~\ref{Veechsurfaces} we prove non-uniform hyperbolicity for a wider class of Veech surfaces, the Veech surfaces of \emph{maximal homological rank} (Definition~\ref{def:homrank}). In addition to algebraically primitive Veech surfaces, this class also includes many geometrically primitive examples which are not algebraically primitive, namely all the primitive Prym eigenforms of genus $3$ and $4$ constructed by C.~McMullen in \cite{McMullen2}, as well as many non geometrically primitive Veech surfaces, for instance all the non-primitive Prym eigenforms (Lemma~\ref{lemma:Prymrank}). The Kontsevich--Zorich cocycle is thus non-uniformly hyperbolic with respect to all 
$\SL$-invariant probability measures given by Prym eigenforms (Corollary~\ref{cor:Prym}). We are very grateful to P.~Hubert who suggested to test our criterion on Prym eigenforms, as a main example of 
geometrically primitive, non-algebraically primitive Veech surfaces.

All canonical $\SL$-invariant absolutely continuous invariant probability measures on connected components of strata of abelian differentials are cuspidal Lagrangian (Lemma~\ref{lemma:density}). 
Our criterion therefore implies the non-uniform hyperbolicity of the Kontsevich--Zorich cocycle with respect to all canonical measures on the moduli space of abelian differentials  (Corollary~\ref{cor:canonic}), a result proved in \cite{Ftwo} (see also \cite{Avila:Viana}). The argument given 
here for this case is in fact a simplified and streamlined version of the original argument.

There are several interesting $\SL$-invariant probability measures to which our criteria 
may be applied. One of the most interesting in our opinion is given by the
algebraic (singular) measures on the moduli space of abelian differentials  coming from canonical measures on strata of quadratic differentials by the standard double (orienting) cover construction.
This application has been carried out recently by R.~Trevi\~no who has derived a proof that
all such measures are non-uniformy hyperbolic \cite{Trevino}. 

\smallskip
This paper is organized as follows. In section~\ref{expformulas} we recall the variational
formulas for the exponents proved in \cite{Ftwo}. In section~\ref{Lab} we prove a transversality
result for the unstable space of the Kontsevich--Zorich cocycle with respect to integral
Lagrangian subspaces in homology. This is a crucial improvement over \cite{Ftwo}.
In section~\ref{asymptotics} we prove following \cite{Ftwo} the key asymptotic formulas
near appropriate boundary points of the moduli space. Section~\ref{Nuh} is devoted
to carrying out the proof of the the main theorem on the non-uniform hyperbolicity
of the Kontsevich--Zorich cocycle. Finally, section~\ref{applications} presents two
fundamental applications. The first, in section~\ref{Veechsurfaces}, is to a wide class
of Veech surfaces, which includes algebraically primitive Veech surfaces already
covered by Bouw-M\"oller results, but it also applies to many other cases. The second,
in section~\ref{canonicmeas}, is a review of the case of canonical measures on connected 
components of strata of abelian differential, already treated in \cite{Ftwo}. The argument we 
present here is somewhat simpler and hopefully provides a template for other cases, such as 
measures coming from strata of quadratic differentials.

\section{Formulas for the exponents}
\label{expformulas}

We recall below for the convenience of the reader, the relevant formulas for  all partial sums 
of the Kontsevich--Zorich exponents. Such formulas were first derived in \cite{Ftwo} (see also 
\cite{Krikorian}, \cite{ForniSurvey}, \cite{VeechBrin}) as a generalization of the Kontsevich--Zorich formula for the sum of all the exponents. The exposition below follows  \cite{Forni:Matheus:Zorich:two} .

Let $S$ be a Riemann surface. The natural Hermitian intersection form on the complex
cohomology $H^1(S,\C)$ of the Riemann surface can be defined on closed 1-forms 
representing cohomology classes as
\begin{equation}
\label{eq:Intform}
(\omega_1,\omega_2):=
\frac{i}{2}\int_S\omega_1\wedge\bar\omega_2\ .
\end{equation}
Restricted to the subspace $H^{1,0}(S,\C)$ of holomorphic $1$-forms,
it induces a positive definite Hermitian form. Hence, by the Hodge
representation theorem, it induces a positive definite bilinear form
on the cohomology $H^1(S,\R)$. The real Hodge bundle $H^1_g$ (over the moduli space
of abelian differentials) is thus endowed with an inner product, called the {\it Hodge inner 
product} and a norm, called the {\it Hodge norm}.

Given a cohomology class $c\in H^1(S,\R)$, let $\omega_c$ be the unique
holomorphic 1-form such that $c=[\re(\omega_c)]$. Define $\ast c$ to be
the real cohomology class $[\im(\omega_c)]$. The Hodge norm $\|c\|$ satisfies
$$
\|c\|^2=\frac{i}{2}\int_S\omega_c\wedge\bar\omega_c\ =
\int_S \re[\omega_c]\wedge\im[\omega_c]\ ,
$$
or, in other words, $\|c\|^2$ is the value of $(c\cdot\ast c)$ on the
fundamental cycle. The operator $c\mapsto\ast c$ on the real cohomology
$H^1(S,\R)$ of a Riemann surface $S$ is called the
\textit{Hodge operator}.

\smallskip
Associated to any  holomorphic 1-form $\omega$ on the Riemann surface $S$, there is a naturally defined space $L_\omega^2(S)$ defined as the completion of smooth functions on $S$ with respect 
to the topology generated by the inner product
\begin{equation}
\label{eq:L2}
\<f,g\>_\omega = \frac{i}{2} \int_S   f \bar g \, \omega\wedge \bar\omega\,, \quad
\text{for any } f, g \in C^\infty(S)\,.
\end{equation}
In other words, the space $L_\omega^2(S)$ is the standard space of square-integrable
functions with respect to the area form $\frac{i}{2} \omega\wedge \bar\omega$ on $S$.

The subspaces $\Cal M_\omega \subset L^2_\omega(S)$ ($\bar{\Cal M}_\omega \subset 
L_\omega^2(S))$ of meromorphic (resp. anti-meromorphic) functions  are both finite dimensional,
of dimension equal to the genus of the surface. In fact, they can be described as follows.
Let $H^{1,0}(S)$ ($H^{0,1}(S)$) be the space of holomorphic (resp. anti-holomorphic) $1$-forms
on $S$. Both have dimension equal to the genus of the surface by Riemann-Roch theorem.
The following characterization of the subspaces $\Cal M_\omega$ and, consequently,
$\bar {\Cal M}_\omega$, holds (see \cite{Fone}, \cite{Ftwo}):
\begin{equation}
\Cal M_\omega = \{ \tilde \omega /\omega \vert  \tilde \omega \in H^{1,0}(S) \} \,.
\end{equation}
We remark the space $\Cal M_\omega$ endowed with the norm of the Hilbert space
$L^2_\omega(S)$ is isometric to the space $H^{1,0}(S)$ endowed with the restriction 
of the hermitian intersection form (\ref{eq:Intform}). In fact,
$$
 \<\frac{\omega_1}{\omega}, \frac{\omega_2}{\omega}\>_\omega =(\omega_1, \omega_2)\,,
\quad \text{ for all } \omega_1\,, \omega_2 \in H^{1,0}(S)\,.
$$
Let $\pi_\omega: L^2_\omega(S) \to \bar{ \Cal M}_\omega$ denote the orthogonal
projection and let $H_\omega$ be the following positive-semidefinite hermitian form 
on $H^{1,0}(S)$:
\begin{equation}
\label{eq:H}
H_\omega(\omega_1, \omega_2):=  \< \pi_\omega (\frac{\omega_1}{\omega}), 
\pi_\omega (\frac{\omega_2}{\omega}) \>_\omega \,, \quad \text {\rm for all }\, 
\omega_1, \omega_2 \in H^{1,0}(S)\,.
\end{equation}
Let $B_\omega$ be the complex bilinear form  on $H^{1,0}(S)$ defined as follows: 
\begin{equation}
\label{eq:B}
B_\omega(\omega_1,\omega_2):=
 \< \frac{\omega_1}{\omega}, \frac{ \bar \omega_2}{\bar\omega}\>_\omega\,, 
\quad  \text{ \rm for all }\, \omega_1, \omega_2 \in H^{1,0}(S).
\end{equation}

The geometric significance of the forms $B_\omega$ and $H_\omega$ is
related to notions of \emph{second fundamental form} and \emph{curvature }of the 
Hodge bundle (see \S 1 in \cite{Forni:Matheus:Zorich:two}). Their significance for the dynamics of 
the Kontsevich--Zorich cocycle  lies in the variational formulas proved in \cite{Ftwo}, \S\S 2-5, 
which give the variation of the Hodge norm on the (real) Hodge bundle $H^1_g$ under the 
action of the Kontsevich--Zorich cocycle (see also \cite{ForniSurvey}, \cite{Forni:Matheus:Zorich:two}).

 The forms $H_\omega$ and $B_\omega$ are in fact related as follows (see \cite{Ftwo}, \S 4). 

 Let $\{\omega_1, \dots, \omega_g\}$ be any orthonormal basis of $H^{1,0}(S)$.
 The restriction of the projection operator $\pi_\omega \vert _{{\Cal M}_\omega}$ 
 is then given by the  following formula:
 \begin{equation}
 \label{eq:proj}
 \pi_\omega (\frac{\tilde\omega}{\omega}) = \sum_{i=1}^g  B_\omega(\tilde\omega, \omega_i)
 {\bar \omega}_i\,, \quad \text{ \rm for all }\, \tilde \omega\in H^{1,0}(S)\,.
 \end{equation}
 It follows that the matrices of $H$ and $B$ of the forms
 $H_\omega$ and $B_\omega$ respectively, with respect to any orthonormal
 basis of the space $H^{1,0}(S)$ are related by the following identity:
 \begin{equation}
 \label{eq:formid}
 H= B B ^\ast\,.
 \end{equation}
 In particular, the forms $H_\omega$ and $B_\omega$ have the same rank and 
  their eigenvalues are related. 
  (The above formula (\ref{eq:formid}) is the correct version of the formula $H= B^\ast B$ which appears as formula $(4.3)$ in \cite{Ftwo} and as formula $(44)$ in \cite{ForniSurvey}. This mistake
there is of no consequence).
  
 Let $ \text{\rm EV}(H_\omega)$ and  $\text{\rm EV}(B_\omega)$ denote the set of eigenvalues of the forms $H_\omega$ and $B_\omega$ respectively. The following identity holds:
  $$
  \text{\rm EV}(H_\omega) = \{ \vert \lambda\vert^2  :  \lambda \in  \text{\rm EV}(B_\omega)\}\,.
 $$   
By the Hodge representation theorem for Riemann surfaces (which
states that any real cohomology class can be represented as the real
or imaginary part of a holomorphic $1$-form), the forms $H_\omega$ and  $B_\omega$ 
on $H^{1,0}(S)$ can be interpreted as bilinear forms, denoted by $H_\omega^\R$ and  
$B_\omega^\R$  respectively, on the real cohomology $H^1(S,\R)$.  The  forms $H_\omega^\R$ and  
$B_\omega^\R$ on $H^1(S,\R)$ have the same rank, which is equal to twice the common 
rank of the forms $H_\omega$ and  $B_\omega$ on $H^{1,0}(S)$. The form
$H_\omega^\R$ is real-valued and positive-semidefinite, while the form $B_\omega^\R$
is complex-valued. For every holomorphic differential $\omega\in \H_g$, the
eigenvalues  of the positive-semidefinite form $H_\omega^\R$ on $H^1(S,\R)$ will be
denoted as follows:
\begin{equation}
\label{eigenvalues}
\Lambda_1(\omega) \equiv 1 \geq \Lambda_2(\omega) \geq \dots \geq \Lambda_g(\omega)\geq 0\,.
\end{equation}
We remark that the above eigenvalues induce well-defined continuous non-negative bounded 
functions on the moduli space of all abelian differentials.

\smallskip
In \cite{Ftwo}, \S2, \S 3 and \S 5 several variational formulas for the Hodge norm on
the real cohomology bundle along  trajectories of the Teichm\"uller flow were proved
and formulas for the Lyapunov exponents of the Kontsevich--Zorich cocycle were derived.
Such formulas generalize the fundamental Kontsevich--Zorich formula for the sum of 
all Lyapunov exponents \cite{Kontsevich}. The formulas are written in \cite{Ftwo} with different 
notational conventions. In fact, as explained above, any abelian holomorphic differential 
$\omega$ on a Riemann surface $S$ induces an isomorphism between the space 
$H^{1,0}(S)$ of all abelian holomorphic differentials on $S$, endowed with
the Hodge norm, and the subspace $\mathcal M_\omega \subset L^2_\omega(S)$ 
of all square integrable meromorphic functions (with respect the area form of the abelian 
differential $\omega$ on $S$). In \cite{Ftwo}, \cite{ForniSurvey} variational formulas
are written in the language of meromorphic functions. We will adopt here the
language of holomorphic abelian differentials.

For any holomorphic abelian differential $\omega$ on a Riemann surface $S$, we defined functions on the Grassmannian $G_k(S,\R)$ of $k$-dimensional isotropic subspaces of $H^1(S,\R)$ as follows. 
Let $I_k \subset H^1(S,\R)$ be any isotropic subspace (with respect to the intersection form)
of dimension $k\in\{1, \dots, g\}$. Let $\{c_1, \dots, c_k\} \subset I_k$ be any  Hodge-orthonormal
basis and let  
$$\{c_1, \dots, c_k, c_{k+1}, \dots,c_g\} \subset H^1(S,\R)
$$ 
be any Hodge-orthonormal Lagrangian completion. Let 
\begin{equation}
\label{eq:Phi_k}
\Phi_k (\omega, I_k) :=  \sum_{i=1}^g \Lambda_i(\omega) - \sum_{i,j=k+1}^g 
\vert B^\R_\omega(c_i, c_j) \vert ^2\,.
\end{equation}
We remark that the above definition is independent of the choice of the orthonormal basis
$\{c_1, \dots, c_k\} \subset I_k$ and of its Hodge-orthonormal Lagrangian completion 
$\{c_1, \dots, c_k, c_{k+1}, \dots,c_g\}$. The function $\Phi_k$ is therefore well-defined
and equivariant under the action of the mapping class group on the Grassmannian bundle 
of the Hodge bundle over the Teichm\"uller space. It induces therefore a function on the 
Grassmannian bundle of the Hodge bundle over the moduli space of all abelian differentials.

\begin{remark}  
The function $\Phi_g$ is the pull-back to the Grassmannian bundle of Lagrangian 
subspaces of the Hodge bundle $H^1_g$ of a function on the moduli space. In other terms, 
it has no dependence on the Lagrangian subspace.  In fact, for every $\omega \in \H_g$ 
and every Lagrangian subspace $I_g \subset H^1(S,\R)$, 
\begin{equation}
\label{eq:Phi_g}
\Phi_g (\omega, I_g) :=  \sum_{i=1}^g \Lambda_i(\omega) \,.
\end{equation}
\end{remark}
This fundamental fact  (discovered in \cite{Kontsevich}) is crucial for the validity of the 
Kontsevich--Zorich formula for the sum of exponents. A version of the formula is given below.

The functions $\Phi_k$ arise in the computation of the hyperbolic Laplacian of the Hodge
norm of isotropic polyvectors  along Teichm\"uller disks.

Let $\{c_1, \dots, c_k\} \subset I_k$ be any Hodge-orthonormal basis of an istropic subspace 
$I_k\subset H^1(S,\R)$ on a Riemann surface $S$. The Euclidean structure defined by the 
Hodge scalar product on $H^1(S,\R)$ defines the natural norm $\Vert c_1 \wedge \dots 
\wedge c_k\Vert_\omega$ of any polyvector which we also call the Hodge norm.
Similarly to the case of the Hodge norm it is defined only by the complex structure of the 
underlying Riemann surface. Thus, for any $(S_0,\omega_0)$ in $\H_g$ the Hodge norm 
$\Vert c_1 \wedge \dots \wedge c_k\Vert_\omega$ defines a smooth function on the
hyperbolic surface obtained as a left quotient $\SO\backslash \SL \omega_0$  of the 
orbit of $\omega_0$. Variation formulas for the Hodge norm of polyvectors were established
in \cite{Ftwo}, Lemmas 5.2 and 5.2':

\begin{lemma} 
\label{lemma:var}
For all $k\in \{1, \dots, g\}$ the following formula holds:
\begin{equation}
\label{eq:var}
\triangle \log \Vert c_1 \wedge \dots \wedge c_k\Vert_\omega =
2 \Phi_k (\omega, I_k)\,.
\end{equation}
\end{lemma}

From the variational formula of Lemma \ref{lemma:var}, it is possible by integration
to derive a formula for the average growth of the spherical averages of the Hodge norm
of any isotropic polyvector on any Teichm\"uller disk.

\begin{lemma} \label{lemma:averagegrowth}
Let $\omega_0$ be an abelian differential on the Riemann surface $S_0$.
For any $k$-dimensional isotropic subspace $I_k \subset H^1(S_0, \R)$ and 
any basis $\{c_1, \dots, c_k\} \subset I_k$,   let $\Vert  c_1 \wedge \dots \wedge c_k\Vert_\omega$ denote the Hodge norm of the polyvector $c_1 \wedge \dots \wedge c_k$ at $(S, \omega)\in \H_g$ 
for any $\omega \in \SO\backslash \SL  \omega_0$. The following formula holds. Let $D_t(\omega_0)$ denote the disk of hyperbolic radius $t>0$ centered at the origin $\SO \omega_0$ 
of the hyperbolic surface $\SO\backslash \SL \omega_0$ , let $\vert D_t\vert$ denote its hyperbolic area and let $d\Cal A_P$ denote the Poincar\'e area element. We have:
\begin{equation}
\label{eq:averagegrowth}
\frac{1}{2\pi} \frac{\partial}{\partial t} \int_0^{2\pi}  
\log \Vert c_1 \wedge \dots \wedge c_k  \Vert_\omega \,d\theta
= \frac{\tanh(t)}{\vert D_t\vert} \int_{D_t(\omega_0)} \Phi_k (\omega, I_k)  d\Cal A_P \,.
\end{equation}
\end{lemma}
\begin{proof} The formula was derived in \cite{Ftwo} as formula 
$(5.10)$. It follows from the variational formula of Lemma \ref{lemma:var} 
by Green formula or explicit integration of the Poisson equation for the hyperbolic Laplacian 
on the Poincar\'e disk (as in \cite{Ftwo}, Lemma 3.1). 
\end{proof}

By the above formula, it is possible to derive formulas for the partial sums of the Lyapunov
exponents of the Kontsevich--Zorich cocyle. Here is our main result.
\begin{theorem} 
\label{thm:partialsum}
Let $\mu$ be any $\SL$-invariant Borel probability ergodic measure on the 
moduli space $\H_g$ of normalized abelian differentials. Assume that  $\lambda^\mu_k > 
\lambda^\mu_{k+1}$, for $k\in \{1, \dots, g-1\}$,  and let $E^+_k$ denote the Oseledec's subbundle carrying the subset $\{\lambda^\mu_1, \dots, \lambda^\mu_k\}$ of the Lyapunov spectrum. Then the following formula holds:
\begin{equation}
\label{eq:partialsum}
\lambda^\mu_1 + \dots +\lambda^\mu_k =
 \int_{\H_g}  \Phi_k\left(\omega,E^+_k(\omega)\right) d \mu(\omega) \,.
\end{equation}
\end{theorem}
\begin{proof} (see \cite{Ftwo}, Corollary 5.5)
For any given $(S,\omega)\in \H_g$, let $\sigma^k_\omega $ denote the normalized canonical (Haar) measure on the Grassmannian $G_k(S,\R)$ of isotropic $k$-dimensional subspaces $I_k\subset 
H^1(S,\R)$, endowed with the euclidean structure given by the Hodge inner product. The
probability measure $\sigma^k_\omega$ is invariant under the action of the circle group $\SO$
on $\H_g$. Let $(t,\theta)\in \R^+\times S^1$ denote the geodesic polar coordinates centered at
the origin $\SO \omega$ on the hyperbolic surface $ \SO \backslash \SL \omega$, let
$\text{dist}_{(t,\theta)}^k$ denote the Hodge distance on the Grassmannian of isotropic 
$k$-dimensional subspaces and let $E^+_k(t,\theta)$ denote the unstable Oseledec's subspace at the abelian differential $\omega_{(t,\theta)} \in  \SO \backslash \SL \omega$.

By Oseledec's theorem, for $\mu$-almost all $(S,\omega) \in \H_g$, almost all $\theta\in S^1$ and $\sigma^k_\omega$-almost all $k$-dimensional isotropic subspaces $I_k\in G_k(S,\R)$,
$$ \lim_{t\to+\infty}  \text{dist}^k_{(t,\theta)} (I_k , E^+_k(t,\theta)\bigr) \,\,=\,\, 0\,\,. $$
Consequently, since the function $\Phi_k$ is bounded and continuous, by Fubini's theorem and the dominated convergence theorem,  for $\mu$-almost all $\omega \in \H_g$,
\begin{equation}
\label{eq:isotropic_conv}
\begin{aligned} 
\lim_{t\to+\infty} 
 \frac{1}{\vert D_t\vert} & \int_{G_k(S,\R)}\int_{D_t(\omega)}
 \vert \Phi_k \left( \omega_{(s,\theta)},I_k\right) \\ & -\,\Phi_k  \left( \omega_{(s,\theta)},E^+_k(s,\theta) \right) \vert \, d\Cal A_P(s,\theta)\, d\sigma^k_\omega \,\,= \,\, 0\,.
 \end{aligned}
 \end{equation}
For any $(S,\omega)\in \H_g$ and any isotropic subspace  $I_k\in G_k(S,\R)$, the Hodge norm
$\Vert c_1 \wedge \dots \wedge c_k\Vert_\omega$ does not depend on the Hodge orthonormal
basis $\{c_1, \dots, c_k\} \subset I_k$, hence it defines a function on the Grassmannian bundle
of $k$-dimensional isotrpoic subspaces of the (real) Hodge bundle. By formula (\ref{eq:isotropic_conv}) and by averaging formula  (\ref{eq:averagegrowth}) over $G_k(S,\R)$ with respect to the measure $\sigma^k_\omega$, then with respect to the  $\SL$-invariant probability measure $\mu$ on $\H_g$ and by applying Fubini's theorem and the dominated convergence theorem, we find that, since $E^+_k$ is an invariant subbundle of  the Kontsevich--Zorich cocycle, the following formula holds:
\begin{equation}
\begin{aligned}
\lim_{t\to +\infty} \frac{\partial}{\partial t}  \int_{\H_g}\int_{G_k(S,\R)}
 &\log  \Vert  c_1 \wedge \dots \wedge c_k \Vert_{\omega_{(t,\theta)}}
 \,d\sigma^k_\omega d\mu(\omega)  \\   &-\, \tanh(t) \int_{\H_g} 
\Phi_k\bigl(\omega,E^+_k(\omega)\bigr)\,d\mu(\omega)\,\,= \,\, 0\,.
\end{aligned}
 \end{equation}
Finally, by averaging over $[0,T]$ with respect to the radial coordinate $t>0$ (Teichm\"uller time) 
and by Oseledec's theorem, we obtain  formula (\ref{eq:partialsum}).

\end{proof}

In the case $k=g$, since the function $\Phi_g$ does not depend on the Lagrangian subspace of
the cohomology, Theorem~\ref{thm:partialsum} reduces to a version of the Kontsevich--Zorich formula for the sum of all non-negative Lyapunov exponents (see \cite{Kontsevich} and \cite{Ftwo}, Corollary 5.3):
\begin{corollary} 
\label{cor:KZformula}
Let $\mu$ be any $\SL$-invariant Borel probability ergodic measure on the 
moduli space $\H_g$ of normalized abelian differentials. The following formula holds:
$$
\lambda^\mu_1 + \dots +\lambda^\mu_g= \int_{\H_g} (\Lambda_1+ \dots + \Lambda_g) d\mu \,.
$$
\end{corollary} 

\section{Lagrangian abelian differentials}
\label{Lab}
The following crucial transversality result holds:

\begin{lemma} 
\label{lemma:transversality}
The unstable bundle $E^+$ of the Kontsevich--Zorich cocycle is $\mu$-almost everywhere
transverse to all integral Lagrangian subspaces,  with respect to any $\SL$-invariant 
ergodic probability measure $\mu$ on a stratum $\H(\kappa)$, that is, for $\mu$-almost all
$(S,\omega) \in \H(\kappa)$ and for any integral Lagrangian subspace $\Lambda \subset H^1(S,\R)$,  
$$
E^+(\omega) \cap \Lambda =\{0\}\,.
$$
\end{lemma}
\begin{proof} Let $\mu$ be any $\SL$-invariant Borel probability measure $\mu$ on $\H_g$.
By Oseledec's and Fubini's theorems, for $\mu$-almost all $(S,\omega)\in \H_g$ 
the Oseledec's stable/unstable spaces $E^\pm (e^{i\frac{\theta}{2}}\omega) \subset H^1(S,\R)$ are 
well-defined for (Lebesgue) almost all $\theta \in S^1$ and the following identity holds:
$$
E^+(e^{i\frac{\theta}{2}} \omega) = E^-( e^{i \frac{\theta+\pi}{2}} \omega) \,,
\quad \text{ \rm for almost all }\theta\in S^1.
$$
It follows that the unstable bundle $E^+$ is $\mu$-almost everywhere transverse
to all integral Lagrangian subspaces if and only if the stable subbundle $E^-$ is.

Let $(S,\omega)\in \H_g$ and let $(t,\theta)\in \R^+\times S^1$ denote the
geodesic polar coordinate centered at the origin $\SO\omega$ on the hyperbolic surface 
$\SO \backslash\SL \omega$. Let $E^-(t,\theta)$ denote the stable Oseledec's space at 
$\omega_{(t,\theta)} \in\SO \backslash \SL \omega$. Let $\Lambda \subset H^1(S, \R)$ be any Lagrangian subspace and
let $\{c_1, \dots, c_g\} \subset \Lambda$ be any basis. By Oseledec's theorem, for $\mu$-almost all 
$\omega\in \H_g$ and for almost all $\theta\in S^1$ the following holds:   

if $\Lambda \cap E^-(0, \theta) 
\not = \{0\}$, then
\begin{equation}
\label{eq:gapone}
\lim_{t\to +\infty}  \frac{1}{t}  \log \Vert c_1 \wedge \dots \wedge c_g \Vert_{\omega_{(t,\theta)}} 
< \lambda^\mu_1 + \dots + \lambda^\mu_g\,.
\end{equation}
Thus, for all $\mu$-almost all $\omega\in H_g$, if the set 
\begin{equation}
\label{eq:non_transverse}
\Cal N^-_\Lambda(\omega):=\{\theta\in S^1\vert \Lambda \cap E^-(0, \theta))\not= \{0\} \} \,
\subset \,S^1
\end{equation}
has positive Lebesgue measure then, by averaging the limit in formula (\ref{eq:gapone}) over the circle
and applying the Lebesgue dominated convergence theorem, we conclude that
\begin{equation}
\label{eq:gaptwo}
\lim_{t\to +\infty}  \frac{1}{2\pi t}  \int_0^{2\pi} 
\log \Vert c_1 \wedge \dots \wedge c_g \Vert_{\omega_{(t,\theta)}} \,d\theta
< \lambda^\mu_1 + \dots + \lambda^\mu_g\,.
\end{equation}
However,  since the function $\Phi_g:  \SO\backslash \H_g \to \R^+$, defined in  formulas (\ref{eq:Phi_k})
and  (\ref{eq:Phi_g}), is bounded, by the pointwise ergodic theorem for $\SL$-balls of A.~Nevo and E.~M.~Stein  \cite{Nevo:Stein} and by the  Kontsevich--Zorich formula (see Corollary \ref{cor:KZformula}),  it follows that, for $\mu$-almost all $(S,\omega) \in \H_g$, 
\begin{equation}
\label{eq:NevoStein}
\lim_{t\to +\infty}\frac{1}{\vert D_t\vert} \int_{D_t(\omega)} \Phi_g \,  d\Cal A_P =
\int_{\H_g}   \Phi_g \, d\mu = \lambda^\mu_1 + \dots + \lambda^\mu_g\,.
\end{equation}
By averaging formula (\ref{eq:averagegrowth}) over $[0,t]$ with respect to the radial coordinate and
by formula (\ref{eq:NevoStein}), we derive that for $\mu$-almost all $(S,\omega) \in \H_g$, for any Lagrangian subspace $\Lambda \subset H^1(S,\R)$ and for any basis $\{c_1, \dots, c_g\} \subset \Lambda$,
\begin{equation}
\lim_{t\to +\infty} \frac{1}{2\pi t }  \int_0^{2\pi}  
\log \Vert c_1 \wedge \dots \wedge c_g  \Vert_{\omega_{(t,\theta)}} \,d\theta
=   \lambda^\mu_1 + \dots + \lambda^\mu_g\,.
\end{equation}
which implies that, for  $\mu$-almost all $(S,\omega)\in \H_g$ and for any 
Lagrangian subspace $\Lambda\in H^1(S,\R)$, the set $\Cal N^-_\Lambda(\omega)$ defined
in formula (\ref{eq:non_transverse}) has zero Lebesgue measure, otherwise the inequality in
formula (\ref{eq:gaptwo}) holds. Since the set of all integral Lagrangian subspaces 
is countable, it then follows by Fubini's theorem that the stable subbundle $E^-$  is $\mu$-almost everywhere transverse to all integral Lagrangian subspaces.  
\end{proof}

\bigskip
\noindent Lemma~8.4 of \cite{Ftwo} is replaced by the following statement.

\begin{lemma} 
\label{lemma:specialpoints}
For every cuspidal Lagrangian $\SL$-invariant ergodic probability measure
$\mu$ on a stratum $\H(\kappa)$ of the moduli space of abelian differentials,
there exists an abelian differential $\omega_0 \in \text{supp}(\mu)$ such that 
\smallskip
\begin{enumerate}
\item The vertical foliation ${\F}^v_{\omega_0}$ is Lagrangian;
\item  the differential $\omega_0$ is a density point of a compact positive measure
set ${\Cal P}_{\kappa} \subset \H(\kappa)$ such that 

\smallskip
$(a)$ all $\omega\in {\Cal P}_{\kappa}$ are Oseledec's regular points for the Kontsevich--Zorich 
cocycle and the unstable subspace $E^+(\omega)$ of the cocycle depend continuously on 
$\omega \in {\Cal P}_{\kappa}$;

\smallskip
$(b)$ the Poincar\'e dual $P({\L}_0) \subset H^1(S,\R)$ of the Lagrangian subspace 
${\L}_0:={\L}({\F}^v_{\omega_0})$, generated by the regular trajectories of the vertical foliation 
${\F}^v_{\omega_0}$, is transverse to $E^+(\omega)$ for all $\omega 
\in {\Cal P}_{\kappa}$.
\end{enumerate}
\end{lemma}
\begin{proof} Let $\omega^- \in \text{\rm supp}(\mu)$ be an abelian differential with Lagrangian 
vertical foliation $\F$. Let $\Cal U \subset \H(\kappa)$ be a neighborhood of $\omega^-$ of 
product type on which the measure $\mu$ has a product structure.
Let $\Cal R_\kappa \subset \H(\kappa)$ be the set of Oseledec's regular points for the
Kontsevich--Zorich cocycle. By Oseledec's and Luzin's theorems and by the transversality Lemma~1, there exists a compact subset $\Cal R^+  \subset \Cal R_\kappa \cap \Cal U$ of positive $\mu$-measure such that the restriction $E^+\vert \Cal R^+$ of the unstable space of the cocycle is continuous and transverse to the Poincar\'e dual $P(\L_{\F})$ of the Lagrangian subspace $\L_{\F} \subset H_1(S,\R)$, generated by the homology classes of the regular leaves of the measured foliation $\F$. 

Let $\omega^+ \in \Cal R^+$ be any density point of $\Cal R^+$ and let $\omega_0$ be any abelian differential in $\Cal W^+_{\Cal U}(\omega^+) \cap \Cal W^-_{\Cal U}(\omega^-)$. By construction the vertical foliation  $\F^v_{\omega_0}$ of the abelian differential $\omega_0 \in \Cal W^-_{\Cal U}(\omega^-)$ is Lagrangian. Let $\Cal K \subset \Cal U$ be a compact neighborhood of $\omega^-$
and let $\Cal P_\kappa:= \Cal W^+_{\Cal U} (\Cal R^+) \cap
 \Cal W^-_{\Cal U}(\Cal K)$. Since the measure $\mu$ has a product structure on $\Cal U$, 
 it follows that for any neighborhood $\Cal V$ of $\omega_0$ in $\Cal U$ the set $\Cal P_\kappa \cap \Cal V$ has positive $\mu$-measure. In fact, there exist neighborhoods $\Cal V^\pm \subset \Cal U$ of  
 $\omega^\pm$ such that 
 $$
 \Cal W^+_\kappa (\Cal V^+) \cap \Cal W^-_\kappa (\Cal V^-)   \subset  \cup_{t\in \R}  g_t (\Cal V)\,.
 $$
By construction, $\mu ( \Cal R^+ \cap \Cal V^+ ) >0$, since $\omega^+$ is a density point of
$\Cal R^+$, and $\mu(\Cal V^-) >0$, since $\omega^- \in \text{ \rm supp}(\mu)$,  hence 
$$
\Cal W^+_{\Cal U} (\Cal R^+ \cap \Cal V^+) \cap \Cal W^-_{\Cal U} (\Cal V^-) \subset 
\cup_{t\in \R} g_t( \Cal P_\kappa \cap \Cal V)
$$ 
has positive $\mu$-measure, which implies that the set $\Cal P_\kappa \cap \Cal V$ has positive
$\mu$-measure by the Teichm\"uller invariance of the measure.

The unstable bundle $E^+$ of the Kontsevich--Zorich cocycle is locally constant along the leaves of 
the  foliation $\Cal W^+_\kappa$ in the following sense. Let $\Cal U\subset \H(\kappa)$ be any open set of product type. For any Oseledec's regular point $\omega \in \Cal U$ and for any $\tilde \omega \in \Cal W^+_{\Cal U}(\omega)$, the unstable space $E^+(\tilde\omega) = E^+(\omega)$.  This property can be derived from the representation theorem (see \cite{Ftwo}, Thm.~8.3), which states that the unstable space $E^+(\omega)$ can be identified with the space of cohomology classes of all basic currents in 
the dual Sobolev space $\H^{-1}_\omega(S)$ for the horizontal foliation 
$\F^h_\omega$.  It follows that the function $E^+\vert \Cal P_\kappa$ is continuous and transversal 
to $P(\L_{\F})$, since by construction $\Cal W^+_{\Cal U}(\Cal P_\kappa) =\Cal W^+_{\Cal U}(\Cal R^+)$ and since the function $E^+\vert \Cal R^+$ is continuous and transverse to $P(\L_{\F})$. 
The argument is complete.

\end{proof}

\section{Asymptotic formulas}
\label{asymptotics}

In this section we compute the Hodge curvature near certain boundary
points of the moduli space. Such boundary points are obtained by pinching a higher
genus surface along the waist curves of cylinders of a Lagrangian foliation, hence they are represented
by Riemann surfaces with nodes made of one or several punctured Riemann spheres carrying 
a meromorphic differential with at least $2g$ paired poles. 

Let $\{a_1,\dots, a_g, b_1, \dots, b_g\}$ be a system of curves defining a canonical basis of the (real)
homology (a marking) of the topological surface $S$ and let $\{(S_\tau,\omega_\tau)\vert \tau \in 
[\C\setminus\{0\}]^{g+s}\}$ be any smooth family of holomorphic differentials (the differential  
$\omega_\tau$ is holomorphic on the Riemann surface $S_\tau$) obtained by pinching the system 
of non-intersecting curves 
$$
\{a_1, \dots, a_g, a_{g+1}, \dots, a_{g+s}\}\,,
$$
 which converges (projectively) as $\tau\to 0$ in $\C^{g+s}$ to a meromorphic differential $\omega_0$ with $2g$ paired poles on a Riemann surface with nodes $S_0$ which is the union of $s$ punctured Riemann spheres.

The following results give the asymptotics of the period matrix and of its derivative in the direction
 of the Teichm\"uller flow as the pinching parameter $\tau \to 0$ along the deformation described
 above. Let us recall that for any Riemann surface $R$ and for any canonical basis of the homology
 $H_1(R,\Z)$, the period matrix $\Pi$ is the $g\times g$ symmetric complex matrix with positive definite imaginary part defined as follows. Let $\{\theta_1, \dots, \theta_g\}$ be the unique basis of holomorphic differentials on $R$, dual to the canonical homology basis $\{[a_1], \dots, [a_g], [b_1], \dots, [b_g]\}$, in the sense that 
 $$
 \int_{a_i} \theta_j =\delta_{ij} \,, \quad \text{ \rm for all }  i,j \in \{1, \dots, g\}\,.
  $$
 The period matrix of the Riemann surface $R$ is then given by the formula:
 $$
 \Pi_{ij} :=   \int_{b_i} \theta_j  \,, \quad \text{ \rm for all }  i,j \in \{1, \dots, g\}\,.
 $$ 
 Let us fix a canonical homology basis $\{[a_1], \dots, [a_g], [b_1], \dots, [b_g]\}$ on the topological
 surface $S$ and  let $\Pi(\tau)$ denote the period matrix of the Riemann surface $S_\tau$
 for all pinching parameters $\tau \in [\C\setminus\{0\}]^{g+s}$.

 Let  $\tau' = (\tau_1, \dots, \tau_g)$ denote the pinching parameters relative to the system
 of curves $\{a_1, \dots, a_g\}$ and let $\tau''= (\tau_{g+1}, \dots, \tau_{g+s})$ be the pinching
 parameters relative to the system of curves $\{a_{g+1}, \dots, a_{g+s}\}$.
 
 \begin{lemma}
 \label{lemma:Piasymptotics}
The family of $g\times g$ complex matrices
\begin{equation}
\label{eq:periods}
\Pi_{ij}(\tau) - {\frac{1}{2\pi \imath}} \delta_{ij} \log \vert \tau'_i \vert 
\end{equation}
is bounded as the pinching parameter $\tau':=(\tau_1,\dots,\tau_g) \to (0,\dots,0)$ in $\C^g$, uniformly with respect to $\tau'' \in \C^{s}$, $\vert \tau''\vert \leq 1$.
\end{lemma}
\begin{proof} The asymptotics follows from classical formulas for the period matrix near the
boundary of the Deligne-Mumford compactification of the moduli space of abelian differentials
(see for instance \cite{Fay}, III, p. 54 or \cite{Yamada}, \S 3, Cor.~6).  A similar formula for the particular 
case $s=0$, which corresponds to boundary points given by meromorphic differentials supported on a 
single punctured Riemann sphere (with $2g$ paired punctures), appears in \cite{Ftwo}, formula 
$(4.33)$.
\end{proof}

The asymptotics of the Lie derivative $d\Pi/d\mu$ of the period matrix $\Pi$ 
in the direction of the Teichm\"uller flow is given by the following result. 

Let $p^\pm_1, p^\pm_2, \dots, p^\pm_g$ denote the $2g$ paired 
punctures on the punctured Riemann surface $S_0$ corresponding to the pinching of the curves 
$a_1, a_2, \dots, a_g$  and let  $\pm \rho_1, \dots, \pm \rho_g \in \C\setminus \{0\}$ denote the 
residues of the limit meromorphic abelian differential $\omega_0$ at $p^\pm_1, \dots, p^\pm_g$ respectively.

\begin{lemma}
\label{lemma:Piderasymptotics}
The family of $g\times g$ complex matrices
\begin{equation}
\label{eq:Piderasymptotics}
\frac{d\Pi}{d\mu}(\tau) - \frac{1}{2\pi} \frac{ \bar \rho_i}{\rho_i} \delta_{ij} \log \vert \tau'_i \vert
\end{equation}
is bounded as the pinching parameter $\tau':=(\tau_1,\dots,\tau_g) \to (0,\dots,0)$ in $\C^g$, uniformly with respect to $\tau'' \in \C^{s}$, $\vert \tau'' \vert \leq 1$. 
 \end{lemma}
 \begin{proof} The proof is based on Lemma~4.2 of \cite{Ftwo} and Lemma~4.2' in \cite{Ftwo}. In that paper, the asymptotic was proved for the particular case $s=0$ of boundary points given by meromorphic differentials supported on a single punctured Riemann sphere (with $2g$ paired punctures). In the more general case, considered here, the limit points in the Deligne-Mumford
 boundary as $\tau' \to 0\in \C^g$ consists of meromorphic differentials on pinched Riemann surfaces $S_0$ with possibly several parts (at most $s+1$) which are punctured Riemann spheres. 
 The asymptotics claimed above follows immediately from \cite{Ftwo}, Lemma~4.2', as
 $\tau'\to 0\in \C^g$ while $\tau'' \in \C^s$ varies within a compact subset of the set
 $$
 D''_0:=  \{ \tau'' \in \C^s\vert \vert \tau''\vert \leq 1\}  \cap
  \bigcap_{i=1}^s \{\tau''\in \C^s \vert \tau_{g+i}\not=0\}\,.
$$ 
Let $J\subset \{g+1, \dots, g+s\}$ and let $\tau''_J = (\tau_j)_{j\in J} \in \C^{\# J}$. By \cite{Ftwo}, Lemma~4.2, the derivative of the period matrix  $d\Pi(\tau)/d\mu$ has well defined asymptotics as $\tau' \to 0 \in \C^g$ and $\tau''_J \to 0 \in \C^{\#J}$ under the condition that $\vert \tau''\vert \leq 1$. By \cite{Ftwo}, Lemma~4.2,  the asymptotics of the matrix $d\Pi(\tau)/d\mu$ as $\tau'\to 0$ does not depend, up to uniformly bounded terms, on $\tau''\in D''_0$ and is given by formula (\ref{eq:Piderasymptotics}). 
Thus the general asymptotics claimed above follows.
   \end{proof}

 We remark that it is also possible to give a more direct proof of Lemma~\ref{lemma:Piderasymptotics}
 based on \cite{Ftwo}, Lemma~4.2, along the lines of the proof of Lemma~4.2'.  The main difference
 lies in the fact that in general case the limit meromorphic differentials are supported in general on pinched Riemann surfaces composed of several punctured Riemann spheres. This feature makes essentially no difference in the calculations carried out in \cite{Ftwo}, Lemma~4.2'.

\begin{corollary}
\label{cor:Piderasymptotics}
 If the residues $\rho_1, \dots, \rho_g$ of the limit meromorphic differential $\omega_0$ on $S_0$
are all real (and non-zero), the following asymptotics holds:
\begin{equation}
\label{eq:Piconverge}
\im\,\Pi(\tau)^{-1/2} \,{\frac{d\Pi}{d\mu}(\tau)}\, \im\,\Pi (\tau)^{-1/2}  \,\, \to  \,\, - I_g \,\, ,  
\end{equation}
as the pinching parameter $\tau':=(\tau_1,\dots,\tau_g) \to (0,\dots,0)$ in $\C^g$, uniformly with respect to $\tau'' \in \C^{s}$, $\vert \tau'' \vert \leq 1$. 
\end{corollary}

Let $P(\Cal A)\subset H^1(M,\R)$ be the Poincar\'e dual of the Lagrangian subspace ${\Cal A}\subset H_1(M,{\R})$ generated by the system $\{[a_1], \dots, [a_g]\}$ and let  $G_{\Cal A}(S,\R)$ denote the open subset of the Grassmannian of  all Lagrangian subspaces $\Lambda \subset H^1(M,\R)$   transversal to the Poincar\'e dual $P(\Cal A) \subset H^1(S,\R)$.  

\begin{lemma} 
\label{lemma:Bconverge}
For any $\Lambda \in G_{\Cal A}(S,\R)$, let $\Cal B^\Lambda_\tau:=
\{c^\Lambda_1(\tau), \dots, c^\Lambda_g(\tau)\}$ be any Hodge orthonormal basis of the 
Lagrangian subspace $\Lambda \subset H^1(S, \R)$  on the Riemann surface $S_\tau$, 
for any $\tau\not =0$. The following limit holds 
\begin{equation}
\label{eq:Bconverge}
B^\R_{\omega_\tau} \left( c^\Lambda_i(\tau), c^\Lambda_j(\tau) \right) \to -\delta_{ij} \,, \quad 
\text{ as }\, \tau \to 0\,,
\end{equation}
uniformly with respect to $\Lambda$ in any given compact subsets of $G_{\Cal A}(S,\R)$ and to any family $\{\Cal B_\tau\}$ of Hodge orthonormal bases for  $\Lambda\in G_{\Cal A}(S,\R)$.
\end{lemma}

\begin{proof}
Let $\{\theta_1(\tau),\dots,\theta_g(\tau)\}$ be the dual basis of holomorphic differentials on the 
Riemann surface $S_\tau$, defined by the standard condition $\theta_i(a_j)=\delta_{ij}$, for all $i$, 
$j\in \{1, \dots, g\}$. Let $\{c^\Lambda_1, \dots, c^\Lambda_g \} \subset H^1(M,\R)$ be a fixed basis 
of the Lagrangian subspace $\Lambda$, represented by the system of harmonic differentials 
$\{\re\,\xi_1(\tau), \dots, \re\,\xi_g(\tau)\}$ on the Riemann surface $S_\tau$. The differentials 
$\xi_1(\tau), \dots, \xi_g(\tau)$ are holomorphic and form a basis of the space of holomorphic 
differentials on $S_\tau$.  As a consequence, there exists a complex $g\times g$ invertible matrix 
$\zeta(\tau):=(\zeta_{ij}(\tau))$ such that, for all $i\in \{1,\dots, g\}$,
\begin{equation}
\label{eq:xi}
\xi_i(\tau) \,=\, \sum_{j=1}^g  \zeta_{ij}(\tau)\, \theta_j(\tau)\,\,. 
\end{equation}
Let $\alpha$, $\beta$ be the $g\times g$ real matrices defined by
\begin{equation}
\label{eq:alphabeta}
\alpha_{ij}:=\int_{a_j} \re\,\xi_i(\tau)  \,\,, \quad  \beta_{ij}:=\int_{b_j} \re\,\xi_i(\tau) \,\,. 
\end{equation}
Since the harmonic differential $\re\,\xi_i(\tau)$ represents the fixed cohomology class $c^\Lambda_i \in 
H^1(S,\R)$, for all $i\in \{1,\dots, g\}$, the matrices $\alpha$ and $\beta$ do not depend on 
$\tau\in \C^{g+s}$. In addition, since $\Lambda \cap P(\Cal A)=\{0\}$, the matrix $\alpha$ is invertible. 
By the definition of the dual basis $\{\theta_1(\tau),\dots,\theta_g(\tau)\}$ and of the period matrix 
$\Pi(\tau)$ of the Riemann surface $S_\tau$ with respect to the canonical homology basis 
$\{[a_1],\dots, [a_g],[b_1],\dots, [b_g]\}$, we have
\begin{equation}
\label{eq:zeta}
\re\,\zeta(\tau) = \alpha  \,\, , \quad  \im\,\zeta(\tau)= \left( \alpha\, \re\,\Pi(\tau) - \beta \right)  
(\im\,\Pi(\tau))^{-1}\,\,. 
\end{equation}
Let $\{ c^\Lambda_1(\tau), \dots, c^\Lambda_g(\tau) \} \subset H^1(M,\R)$ be any orthonormal basis of the subspace $\Lambda$ with respect to the Hodge inner product on $S_\tau$, represented by a orthonormal system of harmonic differentials  $\{ \re\,h_1(\tau), \dots, \re\,h_g(\tau)\}$ on the Riemann surface $S_\tau$. 

There exist  a {\it real }invertible matrix $R(\tau):=(R_{ij}(\tau))$  and  a complex invertible matrix
$Z(\tau):= (Z_{ij}(\tau))$ such that the basis $\{h_1(\tau),\dots, h_g(\tau)\}$ 
can be written in terms of the bases $\{ \xi_1(\tau), \dots, \xi_g(\tau)\}$ and $\{\theta_1(\tau),\dots,
\theta_g(\tau)\}$ of holomorphic differentials by the formulas
\begin{equation}
\label{eq:orthonormalbasis}
h_i(\tau) \,= \,\sum_{j=1}^g  R_{ij}(\tau)\, \xi_j(\tau)  = \sum_{j=1}^g Z_{ij}(\tau)\, \theta_j(\tau)\,.
\end{equation}
(We remark that the matrix $Z(\tau)$  is the inverse of the matrix $C$ of change of basis which appears
in \cite{Ftwo}, formulas  $(4.5)$, $(4.6)$ and $(8.32)$). 

By (\ref{eq:xi}) and (\ref{eq:orthonormalbasis}) it follows that $Z(\tau)= R(\tau)\,
\zeta(\tau)$, hence by (\ref{eq:zeta}),
\begin{equation}
\begin{aligned}
\re\,Z(\tau) &= R(\tau)\,\alpha  \,, \\
 \im\,Z(\tau)&=R(\tau)\, \Bigr(\alpha \re\,\Pi(\tau)- 
\,\beta \Bigr) \im\,\Pi(\tau)^{-1}\,\,. 
\end{aligned}
\end{equation}
Since the matrix $\alpha$ is invertible, we can write
\begin{equation}
\label{eq:ImZ}
\im\,Z(\tau)=\re\,Z(\tau) \, \Bigl( \re\,\Pi(\tau) - \alpha^{-1}\beta \Bigr) \im\,\Pi(\tau)^{-1}\,\,.   
\end{equation}

\noindent Let $B(\tau)$ be the symmetric complex matrix defined as follows:
\begin{equation}
\label{eq:Bdef}
 B_{ij}(\tau):=B^\R_{\omega_\tau}\left(c^\Lambda_i(\tau), c^\Lambda_j(\tau)\right)\,,
\quad \text{ \rm for all } i, j \in \{1, \dots, g\}\,.
\end{equation}
By definition (\ref{eq:B}), the matrix $B(\tau)$ depends on the holomorphic differential 
$\omega_\tau\in \H(\kappa)$ as well as on the Hodge orthonormal basis 
$\{c^\Lambda_1(\tau),\dots,c^\Lambda_g(\tau)\}$ of the Lagrangian subspace $\Lambda
 \in G_{\Cal A}(S, \R)$.

Let  ${{d\Pi}(\tau)/ {d\mu}}$ denote  the derivative of the period matrix in the direction of the Teichm\"uller  flow at the holomorphic differential $\omega_\tau \in \H(\kappa)$. 

We claim that the following formula holds (see also \cite{Ftwo}, formula $(8.32)$):
\begin{equation}
\label{eq:Bmatrix}
B(\tau)= Z(\tau)\,{  \frac{d\Pi}{d\mu}} (\tau) \,Z(\tau)^t\,\,.
\end{equation}
In fact, by Rauch's variational formula, for all $i$, $j \in   \{1, \dots, g\}$, 
$$
B_{\omega_\tau}(\theta_i(\tau),\theta_j(\tau)) =  \frac{\imath}{2} \int_{S} \theta_i(\tau) \,\theta_j(\tau)\,  
  \frac{\bar \omega_\tau}{\omega_\tau} = {\frac{d\Pi_{ij}}{d\mu}}(\tau) \,, 
$$
since the form $B_{\omega_\tau}$ is bilinear, it then follows from the identities 
(\ref{eq:orthonormalbasis}) that
\begin{equation}
\begin{aligned}
B_{ij}(\tau)&=    \sum_{s,t=1}^g  Z_{is}(\tau) Z_{jt}(\tau)
B_{\omega_\tau}(\theta_s(\tau),\theta_t(\tau)) \\
   &=\sum_{s,t=1}^g  Z_{is}(\tau) \, Z_{jt}(\tau) \,  [{\frac{d\Pi_{st}}{d\mu}}(\tau)]\,.
 \end{aligned}
\end{equation}

\smallskip
According to a classical formula \cite{FK}, III.2.3, for any basis $\{\theta_1, 
\dots, \theta_g\}$ dual to a canonical homology basis $\{[a_1], \dots, [a_g], [b_1], \dots, [b_g]\}$, 
$$
{ \frac{\imath}{2}}  \int_M \theta_i \wedge {\overline \theta_j}={ \frac{\imath}{2}} 
  \sum_{s=1}^g [\int_{a_s} \theta_i \int_{b_s} {\overline \theta_j}- 
  \int_{b_s}\theta_i  \int_{a_s}{\overline \theta_j}] = \im\,\Pi_{ij}\,\,,
$$
hence the orthonormality condition on the basis $\{h_1(\tau), \dots, h_g(\tau)\}$ yields
$$
\delta_{ij} ={ \frac{\imath}{2}} 
 \int_S h_i(\tau) \wedge {\overline {h_j(\tau)}}=\sum_{s,t=1}^g Z_{is}
(\tau)\, \im\,\Pi_{st}(\tau)\,{\overline {Z_{jt}(\tau)} }\,\,,
$$
hence $Z(\tau)\,\im\,\Pi(\tau)\,Z(\tau)^{\ast}= I_g$ (equivalent to \cite{Ftwo}, $(4.5)$). It follows that
\begin{equation}
\label{eq:Z}
\begin{aligned}
 \re\,Z(\tau)\,\im\,\Pi(\tau)\, \re\,Z(\tau)^t \,+\,  \im\,Z(\tau)\,\im\,\Pi(\tau)  \,\im\,Z(\tau)^t \,&=\, I_g\,\,; \\
 \im\,Z(\tau) \,\im\,\Pi(\tau) \,\re\,Z(\tau)^t  \,- \, \re\,Z(\tau)\,\im\,\Pi(\tau) \,\im\,Z(\tau)^t  
 \,&=\, 0 \,\,. 
 \end{aligned}
\end{equation}

A calculation shows that the second identity in formula (\ref{eq:Z}) is equivalent to the symmetry of
the matrix $\alpha^{-1}\beta$, a property which will play no role in the argument.  We can rewrite
the first identity in (\ref{eq:Z}) as follows. Let  
$$
\Cal E(\tau):= \im\,\Pi(\tau) ^{-1/2}  \left( \re\,\Pi(\tau) - \alpha^{-1}\beta \right) \im\,\Pi(\tau)^{-1/2}\,.
$$
By formulas (\ref{eq:ImZ}) and (\ref{eq:Z})  we have 
$$
\re\,Z(\tau)\, \im\,\Pi(\tau)^{1/2} \Bigl (I_g +  \Cal E(\tau) \Cal E(\tau)^t \Bigr) \im\,\Pi(\tau) ^{1/2}
\, \re\,Z(\tau)^t \,=\, I_g \,.
$$
 It follows from Lemma~\ref{lemma:Piasymptotics} that the matrix $\Cal E(\tau) \to 0$ as $\tau\to 0$, 
 hence the invertible matrix $\re\,Z(\tau) \im\,\Pi(\tau)^{1/2}$ converges to the compact subgroup 
 $O(g,\R)\subset GL(g,\R)$ of $g\times g$ real orthogonal matrices. Hence, for all $\tau \in
 [\C\setminus\{0\}]^{g+s}$ near $0$, there exist a orthogonal real matrix $O(\tau)\in O(g,\R)$ and a matrix 
 $\Cal E' (\tau)$, which converges to $0$ in $GL(g,\R)$ as $\tau\to 0$, such that  
 \begin{equation}
 \label{eq:ReZasympt}
 \re\,Z(\tau)\,\im\,\Pi(\tau)^{1/2} = O(\tau) \left( I + \Cal E' (\tau)\right)  
\end{equation}
Since $O(g,\R)$ is compact, it follows by (\ref{eq:ImZ}) and (\ref{eq:ReZasympt}) that, as $\tau \to 0$,
\begin{equation}
 \label{eq:ImZasympt}
\im\,Z(\tau)\, \im\,\Pi(\tau)^{1/2}= O(\tau) \left( I + \Cal E' (\tau)\right)  \Cal E(\tau) \to 0 \,.
\end{equation}
By formulas (\ref{eq:ReZasympt}) and  (\ref{eq:ImZasympt}) it follows that
\begin{equation}
\label{eq:ZPi}
 \Cal O(\tau) := Z(\tau)\,\im\,\Pi(\tau)^{1/2}  \to  O(g,\R) \,, \quad  \text{ as } \, \tau \to 0\,.
\end{equation}
Since by formula (\ref{eq:Bmatrix}) the matrix $B(\tau)$ can be written as
\begin{equation}
\label{eq:Btau}
 B(\tau)= \Cal O(\tau)  \im\,\Pi(\tau)^{-1/2} \,{\frac{d\Pi}{d\mu}(\tau)}\, 
 \im\,\Pi(\tau)^{-1/2}   \Cal O(\tau)^t\,,
\end{equation}
the desired asymptotic formula  (\ref{eq:Bconverge}) follows the definition of the matrix $B(\tau)$ 
in formula (\ref{eq:Bdef}), from formula (\ref{eq:ZPi}) and from Corollary~\ref{cor:Piderasymptotics}.

 The convergence in  (\ref{eq:Bconverge}) is uniform since as the Lagrangian subspace $\Lambda$ varies in a compact subset of  $G_{\Cal A}(S,\R)$, the basis $\{c^\Lambda_1,\dots, c^\Lambda_g\}$ of $\Lambda$ (fixed with respect to the pinching parameter $\tau \in [\C\setminus\{0\}]^{g+s}$) can be chosen so that the matrices $\alpha^{-1}$ and $\beta$, defined in (\ref{eq:alphabeta}), are uniformly bounded, hence the matrices $\Cal E(\tau)$ and $\Cal E'(\tau)$ converge to $0$ uniformly as 
 $\tau\to 0$. 
\end{proof}

\section{Non-uniform hyperbolicity}
\label{Nuh}

The proof of the main theorem proceeds by contradiction or contraposition based on the
following corollary of the formulas for the Kontsevich--Zorich exponents given in section 
\ref{expformulas}.

\begin{lemma} 
\label{lemma:vanishing}
Let $\mu$ be a $\SL$-invariant ergodic probability measure on a stratum 
$\H(\kappa)$ of the moduli space of abelian differential. If the Kontsevich--Zorich 
cocycle has $g-k$ zero exponents, that is, if 
$$
\lambda^\mu_{k+1} = \dots= \lambda^\mu_g =0\,, 
$$
then for $\mu$-almost all $\omega\in \H(\kappa)$ and for any Hodge-orthonormal isotropic
system $\{c_{k+1}, \dots, c_g\} \subset H^1(S_\omega, \R)$, Hodge orthogonal and symplectic orthogonal to the $k$-dimensional unstable space $E^+_k(\omega) \subset H^1(S_\omega, \R)$ 
of the cocycle, the following holds:
\begin{equation}
B^\R_\omega (c_i , c_j ) =0 \,, \quad \text{ for all }\, i,j \in \{k+1, \dots, g\}\,.
\end{equation}
\end{lemma}

\begin{proof} By the hypothesis on the exponents, it follows that
$$
\lambda^\mu_1 + \dots + \lambda^\mu_k =  \lambda^\mu_1 + \dots + \lambda^\mu_g\,,
$$
hence by the formula on partial sums of exponents (see Theorem~\ref{thm:partialsum})  and by 
the Kontsevich--Zorich formula (see Corollary~\ref{cor:KZformula}), it follows that
\begin{equation}
\label{eq:inteq}
\int_{\H_g}  \Phi_k (\omega, E^+_k(\omega)) d\mu(\omega)= 
\int_{\H_g}  \left(\Lambda_1(\omega) + \dots + \Lambda_g(\omega)\right) d\mu(\omega)\,.
\end{equation}
By the definition of the functions $\Phi_k$ in formula (\ref{eq:Phi_k}), it follows that
$$
 \Phi_k (\omega, E^+_k(\omega)) \leq  \Lambda_1(\omega) + \dots + \Lambda_g(\omega)
$$
hence by the integral identity (\ref{eq:inteq}) equality holds $\mu$-almost everywhere.
It follows that for $\mu$-almost all $\omega\in \H(\kappa)$ and for any Hodge orthonormal
isotropic system $\{c_{k+1}, \dots, c_g\} \subset H^1(S_\omega, \R)$, Hodge orthogonal and
symplectic orthogonal to $E^+_k(\omega) \subset H^1(S_\omega, \R)$, we have:
$$
\sum_{i,j=k+1}^g \vert B^\R_\omega(c_i,c_j) \vert^2 =0\,.
$$
which immediately implies the desired conclusion.
\end{proof}

We finally prove our main result.

\begin{proof} [Proof of Theorem~\ref{thm:main} ]

Since the measure $\mu$ on $\H(\kappa)$ is cuspidal Lagrangian, there is a holomorphic
abelian differential $\omega_0 \in \text{ \rm supp}(\mu)$ such that the properties $(1)$ and
$(2)$ stated in Lemma~\ref{lemma:specialpoints} hold.

Let $\{\omega_t \vert t\geq 0 \}$ denote the forward Teichm\"uller orbit of the holomorphic differential
$\omega_0$ on $S_0$. For each $t\geq 0$, the differential $\omega_t$  is holomorphic on a (unique) Riemann surface $S_t$. Since $\omega_0$ has Lagrangian vertical foliation, the \emph{marked}
abelian differential $\omega_t$ converges projectively as $t\to +\infty$ to a meromorphic differential
on a union of punctured Riemann sphere  with poles at all (paired) punctures, obtained by pinching
the waist curves $\{a_1, \dots, a_{g+s}\}$ of all cylinders of the vertical foliation $\F^v_{\omega_0}$
on $S_0$ (see \cite{Masur1}, Theorem 3).

Let $\Cal A:= \L(\F^v_{\omega_0}) \subset H_1(S_0, \R)$ be the Lagrangian subspace 
generated by the system $\{[a_1], \dots, [a_{g+s}]\} \subset H_1(S_0, \R)$.  By Lemma~\ref{lemma:specialpoints}, the unstable manifold $E^+(\omega_0)$ is well-defined and transverse
to the Poincar\'e dual $P(\Cal A) \subset H^1(S_0, \R)$. Thus there exists a Lagrangian 
subspace $\Lambda_0\subset  H^1(S_0, \R)$  such that $E^+(\omega_0) \subset \Lambda_0$
and $\Lambda_0 \cap P(\Cal A) = \{0\}$, that is $\Lambda_0 \in G_{\Cal A}(S_0, \R)$.

By Lemma~\ref{lemma:Bconverge}, for any compact subset $\Cal G_0 \subset G_{\Cal A}(S_0, \R)$, there exists $t(\Cal G_0)>0$ such that, for all $t\geq t(\Cal G_0)$, for all $\Lambda \in  \Cal G_0$ and for any Hodge orthonormal basis $\{c_1^\Lambda(t), \dots,  c_g^\Lambda(t)\}$ of $\Lambda\subset 
H^1(S_0, \R)$,
$$
\vert B^\R_{\omega_t} \left( c_i^\Lambda(t), c_j^\Lambda(t) \right) +\delta_{ij}  \vert  \leq 1/4\,.
$$
Let  $\Cal G_0 \subset G_{\Cal A}(S_0, \R)$ be any given compact neighborhood of the
Lagrangian subspace $\Lambda_0\subset H^1(S_0, \R)$. Fix any $t>t(\Cal G_0)$. By continuity 
of the Hodge product and of the form $B_\omega$ with respect to the abelian differential 
$\omega \in \H(\kappa)$, there exists a neighboorhood $\Cal V_t \subset \H(\kappa)$ such that, 
for all $\Lambda \in  \Cal G_0$ and for any Hodge orthonormal basis 
$\{c_1^\Lambda(\omega), \dots,  c_g^\Lambda(\omega)\}$ of $\Lambda\subset H^1(S_0, \R)$,
$$
\vert B^\R_{\omega} \left( c_i^\Lambda(\omega), c_j^\Lambda(\omega) \right) +\delta_{ij}  \vert \leq 1/2\,.
$$
Let $\Cal V$ be a neighborhood of $\omega_0$ in $\H(\kappa)$ such that  $g_t(\Cal V) \subset 
\Cal V_t$ and such that for any $\omega \in \Cal V \cap \Cal P_\kappa$, the unstable subspace
$E^+(\omega)$ is contained in a Lagrangian subspace $\Lambda_\omega \in \Cal G_0$.
Since $\omega_0$ is a density point of $\Cal P_\kappa$, by the Teichm\"uller invariance of
the measure, the set $\Cal P_\kappa(t) := g_t (\Cal P_\kappa \cap \Cal V)$ has positive
$\mu$-measure. By construction $g_t (\Cal P_\kappa \cap \Cal V)$ and, for all
$\omega \in \Cal P_\kappa\cap \Cal V$, the unstable space $E^+(g_t\omega) = E^+(\omega) 
\subset  \Lambda_\omega \in \Cal G_0$. It follows that for all $\omega \in \Cal P_\kappa(t)$
and for any Hodge orthonormal basis $\{c_1(\omega) ,\dots c_g(\omega) \}$ of the Lagrangian
subspace $\Lambda_\omega \subset H^1(S_0, \R)$, 
\begin{equation}
\label{eq:Bpositive}
\vert B^\R_{\omega} \left( c_i(\omega), c_i(\omega) \right)  \vert  \geq 1/2 \,,
 \quad \text{ for all } \, i\in \{1, \dots, g\}\,.
\end{equation}
Let us assume that there exists $k<g$ such that $\lambda^\mu_k > \lambda^\mu_{k+1}=0$.
It follows from Lemma~\ref{lemma:vanishing} that, for $\mu$-almost all $\omega \in \H(\kappa)$
and for \emph{any }Hodge orthonormal isotropic system $\{c_{k+1}, \dots, c_g\} \subset H^1(S_\omega, \R)$, Hodge orthogonal and symplectic orthogonal to $E^+(\omega) \subset H^1(S_\omega, \R)$,
\begin{equation}
\label{eq:Bzero}
B^\R_{\omega}( c_i, c_i)  =0\,,  \quad \text{ for all } \, i\in \{k+1, \dots, g\}\,.
\end{equation}
Since by construction $E^+(\omega) \subset \Lambda_\omega$ for all $\omega \in \Cal P_\kappa(t)$,  there exists a Hodge orthonormal basis $\{c_1(\omega) ,\dots, c_g(\omega) \}$ such that 
$\{c_1(\omega) ,\dots, c_k(\omega) \}$ is a basis of $E^+(\omega)$. 
However, since $\Cal P_\kappa(t)$ has positive $\mu$-measure, both the estimate (\ref{eq:Bpositive})
and the identity  (\ref{eq:Bzero}) hold on $\Cal P_\kappa(t)$ leading to a contradiction.
It follows that all the Kontsevich--Zorich exponents are all non-zero.
\end{proof}

For the sake of completeness we sketch below the proof of Theorem~\ref{thm:partial}
following the argument given in \cite{Ftwo}, Cor.~5.4, and \cite{Forni:Matheus:Zorich:two}.

\begin{proof}[Proof of Theorem~\ref{thm:partial}]
If $\mu$ is Lagrangian, there exists $\omega\in \text{\rm supp}(\mu)$ such that the form
$B^\R_{\omega}$ has maximal rank. In fact, let $\omega_0\in  \text{\rm supp}(\mu)$, 
be a holomorphic abelian differential with Lagrangian vertical foliation. The marked
abelian differential $g_t \omega$ converges projectively to a meromorphic differential
on a union of punctured Riemann spheres with poles at all (paired) punctures. It follows
from Lemma~\ref{lemma:Bconverge} that for $t>0$ sufficiently large the form $B^\R_{g_t\omega}$
has maximal rank. By continuity there exists an open set $\Cal U\subset \H(\kappa)$ of positive
$\mu$-measure such that $B^\R_{\omega}$ has maximal rank for all $\omega \in \Cal U$.
By an elementary linear algebra argument, since $B^\R_{\omega}$ has maximal rank, 
the existence of a Hodge orthonormal Lagrangian system $\{ c_1, \dots, c_g\}
\subset H^1(S,\R)$ such that
$$
B^\R_{\omega}(c_i, c_j) =0\,, \quad \text{ \rm for all } i,j \in \{r+1, \dots, g\}\,,
$$
implies that  $r \geq g/2$. Thus the result follows from Lemma~\ref{lemma:vanishing}.
\end{proof}

\section{Fundamental applications}
\label{applications}

\subsection{Veech surfaces}
\label{Veechsurfaces}

The case of the Veech surfaces ($n$-gons) found by W.~Veech in \cite{Veech3} was one of the
original motivation of the construction in \cite{Ftwo}, in particular it inspired the notion of a 
Lagrangian measured foliation and the related focus on meromorphic abelian differentials on 
spheres with $2g$ paired punctures at the boundary of the moduli space. P.~Hubert has 
recently explained to the author that the relevant properties of the Veech $n$-gons of
\cite{Veech3} are in fact shared by a larger class of Veech surfaces, called \emph{algebraically
primitive }Veech surfaces. For all $\SL$-invariant probability measure associated to 
algebraically primitive Veech surfaces, the non-uniform hyperbolicity of the Kontsevich--Zorich
cocycle follows from the formulas of Bouw-M\"oller \cite{Bouw:Moeller}  for individual exponents.
We remark that from the Bouw-M\"oller formulas it seems possible to construct examples 
of non-uniformly hyperbolic $\SL$-invariant probability measures with multiple exponents.

We briefly recall below basic definitions concerning Veech surfaces.

Let $\omega$ be an abelian holomorphic differential on a Riemann surface $S$ 
(or equivalently, let $(S, \omega)$ be a translation surface). The $\SL$-orbit of $\omega$ 
in the Teichm\"uller space of abelian differentials is called its \emph{Teichm\"uller disk}
(it is in fact isomorphic to the unit tangent bundle of a Poincar\'e disk). 

The stabilizer $\operatorname{SL}(\omega) < \SL$ of $\omega$  is a Fuchsian group, called the 
\emph{Veech group}. A translation surface $(S,\omega)$ is called a \emph{Veech surface} if and 
only if the Veech group $\operatorname{SL}(\omega)$ is a lattice in $\SL$. In \cite{Veech3}
Veech proved that if the Veech group is a lattice a dichotomy holds for the directional flows of 
the translation surface: a directional flow is either uniquely ergodic or it is completely
periodic.  

By definition, the $\SL$-orbit of a Veech surface in the moduli  space $\H_g$ of abelian differentials is isomorphic to an immersed hyperbolic surface of finite volume. The (normalized) canonical hyperbolic measure supported on the $\SL$-orbit of a Veech surface in $\H_g$
is a fundamental example of a $\SL$-invariant probability measure on $\H_g$.

An important invariant of a Veech surface is the \emph{trace field }of its Veech group,
that is, the group generated by all the traces of elements of the Veech group. By a theorem
of R.~Kenyon and J.~Smillie \cite{Kenyon:Smillie}, the trace field of a translation surface of genus 
$g\geq 1$ has degree at most $g$ over $\Q$. 

\begin{definition} (see \cite{Moeller1}, \S 2)
A Veech surface $(S,\omega)$ of genus $g\geq 1$ is called \emph{algebraically primitive }if the 
trace field of its Veech group $\operatorname{SL}(\omega)$ has exactly degree $g$ over $\Q$,
that is, it has maximal degree. 
\end{definition}

We recall that a Veech surface is called \emph{geometrically primitive }if it is not a branched cover 
over another Veech surface (of lower genus). While any algebraically primitive Veech surface is geometrically primitive, the converse is not true in genus $g\geq 3$ (see \cite{Moeller1}, \S 2, and 
\cite{McMullen2}).

\smallskip

We introduce below a  class of Veech surfaces, which includes all algebraically primitive Veech surfaces 
and many examples of geometrically, but not algebraically, primitive Veech surfaces (as well as many
non-primitive surfaces)  to which our criterion applies.

\begin{definition} A translation surface $(S,\omega)$ will be called \emph{decomposable }if the set $\mathcal P(\omega)\subset {\mathbb P}^1(\R)$ of completely periodic directions has at least two distinct elements. It is called a \emph{prelattice surface} (or a \emph{bouillabaisse surface}, see \cite{Hubert:Lanneau}) if its Veech group $\operatorname{SL}(\omega)$ contains two transverse parabolic elements. Any prelattice surface has two completely periodic transverse directions (and in each direction the surface split as a union of flat cylinders with commensurable moduli \cite{Veech3}). In particular, it is decomposable.
\end{definition}

Let $(S,\omega)$ be a decomposable translation surface. For any pair of distinct (transverse) directions $\alpha$, $\beta \in \mathcal P(\omega)$,  the \emph{intersection matrix }$E^\omega(\alpha,\beta)$ is defined as follows (see  \cite{Veech3}, \cite{Hubert:Lanneau}).  Let $\{C^\alpha_i \vert i\in \{1, \dots, r\}\}$ and $\{C^\beta_i \vert i\in \{1, \dots, s\}\}$ be the families of all flat cylinders in the direction $\alpha$ and $\beta\in \mathcal P(\omega)$ respectively. The matrix $E^\omega(\alpha,\beta)$ is the $r\times s$ non-negative integer matrix such that $E^\omega_{ij}(\alpha,\beta)$ is equal to the number of parallelograms in the intersection $C^\alpha_i \cap C^\beta_j$. In other terms, the non-negative integer $E^\omega_{ij}(\alpha,\beta)$ is equal to the (algebraic) intersection number of the (positively oriented) waist curves 
$\gamma^\alpha_i$ and $\gamma^\beta_j$ of the cyclinders  $C^\alpha_i$ and $C^\beta_j$ respectively.

\begin{definition} 
\label{def:homrank}
The \emph{homological rank }$r(\omega) \in \{0, 1,\dots g\}$ of a decomposable translation surface $(S,\omega)$ of genus $g\geq 1$ is the integer defined as follows: 
$$
r(\omega) :=  \max \{ \text{ \rm rank } E^\omega(\alpha,\beta) \vert  \alpha\not= \beta \in \mathcal P(\omega) \}\,.
$$
\end{definition}

We remark that it is quite immediate to prove that the homological rank of a translation surface
is at most equal to its genus. In fact, the waist curves of any cylinder decomposition generate
an isotropic subspace of the homology with respect to the intersection form. Since the dimension
of isotropic subspace is at most equal to the genus of the surface, our conclusion follows.

We can thus state the main application of our criterion to Veech surfaces:

\begin{theorem} 
\label{thm:KZexphomrank}
Let $\mu_\omega$ denote the unique $\SL$-invariant probability measure on the moduli space 
$\H_g$ supported on the $\SL$-orbit of a Veech surface $(S, \omega)\in \H_g$. If the Veech surface has maximal homological rank, then the Kontsevich--Zorich cocycle is non-uniformly hyperbolic. In fact, 
the following general statement holds. Let $r:=r(\omega)\in \{1, \dots, g\}$ denote the homological rank 
of a Veech surface $(S, \omega)$. The Kontsevich--Zorich exponents satisfy the inequalities:
$$
\lambda^{\mu_\omega}_1=1 > \lambda^{\mu_\omega}_2 \geq \dots \geq  \lambda^{\mu_\omega}_r >0\,.
$$
\end{theorem}
\begin{proof} It is immediate to verify that the measure $\mu_\omega$ has a local product
structure and it is cuspidal. By definition its homological dimension is at least equal to the homological rank of the Veech suface $(S,\omega)$. In fact, if $\alpha \in \mathcal P(\omega)$ is a completely
periodic direction of homological dimension at most $d\in \{1, \dots, g\}$, then by definition 
the rank of the intersection matrix $E^\omega(\alpha, \beta)$ is at most $d$, for all $\beta\not=\alpha
\in \mathcal P(\omega)$. The statement then follows from Theorem~\ref{thm:main}, in case
the homological rank is maximal, and from Theorem~\ref{thm:general} in general. 
\end{proof}

We are very grateful to P.~Hubert for telling us about the following result from \cite{Hubert:Lanneau},
  \S 2.5. We sketch the argument for the convenience of the reader. 
  
  \begin{lemma} 
  \label{lemma:bouillabaisse}
  Any prelattice (bouillabaisse) algebraically primitive translation surface has maximal 
  homological rank. In fact, for any prelattice surface the rank of the trace field 
  is at most equal to its homological rank.
  \end{lemma}
\begin{proof}  Let $\alpha\not=\beta \in \mathcal P(\omega)$ be two transverse parabolic directions
of a prelattice surface $(S,\omega)$. 
Let us denote the width and heights vectors of the cylinders $C^\alpha_i$ and $C^\beta_j$ 
respectively by $(w^\alpha_i, h^\alpha_i)$ and $(w^\beta_j, h^\beta_j)$. We remark that by 
construction the width vectors $w^\alpha_i$ and $w^\beta_j\in \R^2$ have directions $\alpha$ and $\beta \in P^1(\R)$ respectively for all $i\in \{1, \dots, r\}$, $j\in \{1, \dots,s\}$ . 

Let $E:=E^\omega(\alpha, \beta)$ be the intersection matrix. Let 
\begin{equation}
\begin{aligned}
x &:= (\vert w^\alpha_1\vert, \dots, \vert w^\alpha_r\vert)\,, \quad y:=(\vert h^\alpha_1\vert, \dots, 
\vert h^\alpha_r\vert) \,, \\
\xi &:= (\vert w^\beta_1\vert, \dots, \vert w^\beta_s\vert)\,, \quad \eta:=(\vert h^\beta_1\vert, \dots, 
\vert h^\beta_s\vert)\,.
\end{aligned}
\end{equation}
By construction the following identities hold:
\begin{equation}
\begin{cases}  
\label{eq:xeta} 
x= E \xi\,, \\  \eta= E^t y\,. 
\end{cases}
\end{equation}
 Up to taking a power of the parabolic elements, one can assume that the parabolic elements 
 $P^\alpha$, $P^\beta\in \operatorname{SL}(\omega)$ corresponding respectively to the parabolic directions $\alpha$, $\beta \in \mathcal P (\omega)$, are each a multiple of the Dehn twist of each  cylinder $C^\alpha_i$, $C^\beta_j$ for all $i\in \{1, \dots,r\}$, $j\in \{1, \dots,s\}$. Under this assumption there exist numbers $a$, $b\in \R^+$ such that, with respect to a system of cordinates with axis parallel to the directions $\{\alpha, \beta\}$, 
$$
P^\alpha= \begin{pmatrix} 1 & a \\ 0 & 1\end{pmatrix}  \quad \text{\rm and} \quad
P^\beta= \begin{pmatrix} 1 & 0 \\ b & 1\end{pmatrix} 
$$
By construction all the ratios $x_i/y_i$ are commensurable with $a$ and  all the rations $\xi_j/\eta_j$
are commensurable with $b$, that is, there exist integer vectors $(m_1, \dots, m_r)\in \Z^r$ and $(n_1, \dots, n_s)\in \Z^s$ such that
$$
\begin{cases}  
m_i x_i= a y_i \,,& \quad \text{ for all }i\in \{1, \dots, r\}\,; \\ 
n_j\eta_j =b \xi_j \,,& \quad \text{ for all }j\in \{1, \dots, s\}\,;
\end{cases}
$$
Let $D_m:= \text{\rm Diag}(m_1, \dots, m_r)$ and $D_n:= \text{\rm Diag}(n_1, \dots, n_s)$. The above
equations can be written in matrix fom as
\begin{equation}
\label{eq:mn}
\begin{cases}   
D_m x= a y\,, \\  
D_n\eta= b y\,. 
\end{cases}
\end{equation}
By equations (\ref{eq:xeta}) and (\ref{eq:mn}) it follows after some calculations
\begin{equation}
\label{eq:eigenvector}
\begin{cases}   
ED_n (E^t)D_m x = (ab) x \,, \\
(E^t) D_m E D_n \eta =(ab) \eta \,.
\end{cases}
\end{equation}
It follows that the real number $t:=ab$ is an eigenvalue of the matrices $ED_n (E^t)D_m$
and $(E^t) D_m E D_n$. Since the rank of both the above matrices is less or equal than
the rank of the intersection matrix $E$, which by definition is less or equal to the homological
rank $r(\omega)$ of the translation surface, it follows that $t\in \R$ is an algebraic number of 
degree at most equal to $r(\omega)$. Finally, it is proved in \cite{Hubert:Lanneau}, Claim
2.1, that the trace field of the Veech group $\operatorname{SL}(\omega)$ is equal to
$\Q[t]$, hence its degree is bounded above by the homological rank $r(\omega)$. The
argument is thus completed.
\end{proof}

By Theorem~\ref{thm:KZexphomrank} and Lemma~\ref{lemma:bouillabaisse} we can prove the
following result, which can also be derived from the formulas of I.~Bouw
and M.~M\"oller for single Kontsevich--Zorich exponents (see \cite{Bouw:Moeller}, Thm.~8.2 and Cor.~8.3): 

\begin{corollary} 
\label{cor:KZexpalgrank}
Let $\mu_\omega$ denote the unique $\SL$-invariant probability measure on the moduli space 
$\H_g$ supported on the $\SL$-orbit of a Veech surface $(S, \omega)\in \H_g$. If the Veech surface is algebraically primitive, then the Kontsevich--Zorich cocycle is non-uniformly hyperbolic. In fact, 
the following general statement holds. Let $r\in \{1, \dots, g\}$ denote the rank of the trace
field of the Veech group $\operatorname{SL}(\omega)$ of a Veech surface $(S, \omega)$. The Kontsevich--Zorich exponents satisfy the inequalities:
$$
\lambda^{\mu_\omega}_1=1 > \lambda^{\mu_\omega}_2 \geq \dots \geq  \lambda^{\mu_\omega}_r >0\,.
$$
\end{corollary}

An important family of geometrically primitive Veech surfaces, not algebraically primitive, 
given by \emph{Prym eigenforms} in genus $3$ and $4$, was discovered by C.~McMullen 
\cite{McMullen2}. The Prym class contains an example discovered earlier by M.~M\"oller 
\cite{Moeller1},\S 2 (studied earlier in \cite{Hubert:Schmidt}). In genus $3$ it appears that
most geometrically primitive Veech surfaces are not algebraically primitive. In fact,
conjecturally there are only finitely many algebraically primitive Veech surfaces (in fact, M.~Bainbridge and M.~M\"oller \cite{Bainbridge:Moeller} have recently proved a finiteness result for the stratum $\H(3,1)$), while there are infinitely many geometrically, but not algebraically, primitive Veech surfaces (the 
geometrically primitive Prym eigenforms). 

All Prym eigenforms have a quadratic trace field, hence they are not algebraically primitive
in genus $3$ and $4$. However the following result holds:

\begin{lemma} 
\label{lemma:Prymrank}
All Prym eigenforms (geometrically primitive or not) in genus $g=2$, $3$ and $4$ 
have maximal homological rank.
\end{lemma}
\begin{proof}
The intersection matrices for the Prym eigenforms are computed in \cite{McMullen2}
(see for instance Figure 1 which gives the corresponding Coxeter graphs).
The result is as follows.  Let $E^P_g$ denote intersection matrices for the Prym eigenforms 
in genus  $g\in \{2,3, 4\}$. The following formulas hold:

$$
E^P_2=\begin{pmatrix} 1 & 0\\ 1 & 1\end{pmatrix}\,, \quad 
E^P_3=\begin{pmatrix} 0 & 0 & 1 \\ 0 & 1 & 0 \\ 1 & 1 & 0 \end{pmatrix}\,, \quad 
E^P_4= \begin{pmatrix} 0 & 0 & 1 & 0 \\  1 & 1 & 1 & 0 \\ 0 & 1 & 1 & 1 \\ 
                                             0 & 1 & 0 & 0 \end{pmatrix}\,.
$$
It is quite immediate to verify that the above matrices have maximal rank.

\end{proof}

\begin{corollary} 
\label{cor:Prym}
The  Kontsevich--Zorich cocycle is non-uniformly hyperbolic with respect to the unique 
$\SL$-invariant probability measure $\mu_\omega$ on the moduli space $\H_g$ supported 
on the  $\SL$-orbit of any Prym eigenform $\omega \in \H_g$ in genus $g=2$, $3$ or $4$.
\end{corollary}

\subsection{Canonical measures}
\label{canonicmeas}

In the proof of the non-uniform hyperbolicity of the Kontsevich--Zorich cocycle with respect to
the canonical absolutely continuous $\SL$-invariant probability measures on connected components 
of strata of the moduli space of abelian differentials, a key result is the following density statement 
(see \cite{Ftwo}, Lemma~4.4):

\begin{lemma}
\label{lemma:density}
The subset of Lagrangian abelian differentials is dense in every stratum $\H(\kappa)\subset
 \H_g$ of the moduli space of abelian differentials.
\end{lemma}
\begin{proof} Let $\F(\kappa)$ be the set of isotopy classes of all orientable measured 
foliations which can be realized as horizontal (or vertical) foliation of an abelian differential in $\H(\kappa)$. The multiplicative group $\R^+$ of non-zero real numbers acts on $\F(\kappa)$.
Let $\left(\F(\kappa)\times \F(\kappa)\right)/\R^+$ denote the quotient of the product space $\F(\kappa)\times \F(\kappa)$ with respect to the diagonal action.

The map $\H(\kappa) \to \left(\F(\kappa)\times \F(\kappa)\right)/ \R^+$ defined as
$$
\omega \to [(\F^h_\omega, \F^v_\omega)] \in 
\left (\F(\kappa)\times \F(\kappa)\right) /\R^+ \,, \quad \text{ \rm for all } \omega \in \H(\kappa)\,,
$$
is  locally well-defined and open on the stratum $\Cal H(\kappa)$.

We claim that the set of Lagrangian foliation is dense in $\F(\kappa)$. The proof of
the claim will conclude the argument.

Let $\F := \{ \eta_{\F}=0\} \in \F(\kappa)$ be an orientable measure foliation 
on a surface $S$ of genus $g\geq 1$ and let $\Sigma\subset  S$ be the subset of its singular points. 
It follows from Poincar\'e recurrence theorem that $\F$ is completely periodic whenever the relative cohomology class $[\eta_{\F}] \in \R\cdot H^1(S, \Sigma_{\F}, \Z)$, hence
in particular whenever $[\eta_{\F}] \in \R\cdot H^1(S, \Sigma_{\F}, \Q)$. It follows
that completely periodic foliations are dense in $\F(\kappa)$. In fact, by A.~Katok local classification theorem, the relative period map
$$
\F \to [\eta_{\F}] \in H^1(S, \Sigma_{\F}, \R)
$$
is local homeomorphism on the space $\F(\kappa)$.

Let ${\F}$ be  completely periodic. In this case, the surface can be decomposed, by 
cutting along  the singular leaves, into a finite union of cylindrical components whose number 
is at most $3g-3$. Let $P:H_1(S,{\R})\to H^1(S,{\R})$ be the (symplectic) map
given by the Poincar\'e duality. We claim that, if $\F$ is completely periodic, then
$P^{-1}[\eta_\F]\in \L(\F)$. More precisely, if $\{a_1,\dots,a_s\}$ are the oriented waist 
curves of the cylinders $\{A_1,\dots,A_s\}$ of ${\F}$, which are respectively of transverse 
heights $\{h_1,\dots, h_s\}$, then
\begin{equation}
\label{eq:Poincaredual}
P^{-1}[\eta_{\F}]\,=\, \sum _{i=1}^s h_i\,[a_i] \,\in\, H_1(S,{\R})\,. 
\end{equation}
In fact, if $\gamma\subset S$ is any simple oriented closed curve, then $\gamma 
\cap A_i$ is homologous to $([a_i]\cap [\gamma])\cdot v_i$  relative to $\partial A_i$,
where $v_i$ is a positively oriented vertical segment joining the ends of $A_i$. Hence 
formula (\ref{eq:Poincaredual}) follows.
\smallskip
\noindent  Let $\mathcal L(\F) \subset H_1(S,\R)$ denote the isotropic subspace generated by 
the homology classes of the regular leaves of $\F$ and let $d(\F):= \text{\rm dim}\,
{\mathcal L}({\F})\in \{1,\dots,g\}$.  If $d(\F)=g$, then ${\F}$ is a Lagrangian measured foliation.
Let us assume $d(\F):=d<g$ and let us construct an arbitrarily small perturbation $\F'$ of the foliation
$\F$ such that $d(\F') > d$. 

Let $\{a_1,\dots,a_d\}$ be a maximal system of regular leaves of ${\F}$ such that 
the system of homology classes $\{[a_1],\dots,[a_d]\}$ is linearly independent in $H_1(S,{\R})$,
hence it is a basis of the isotropic subspace $\L(\F) \subset H^1(S, \R)$. Since $d<g$, there
exists a smooth closed curve $\gamma\subset S$ such that $[\gamma] \not\in \L(\F)$, 
\begin{equation}
\label{eq:gamma}
\gamma \cap a_j =\emptyset \,, \,\, \text{ \rm for all } j\in \{1, \dots, d\}\,, \quad
\text{\rm and } \quad \gamma \cap \Sigma_{\F} =\emptyset\,.
\end{equation}
The existence of a curve $\gamma \subset S$ with the above properties can be proved as follows. 
Since $d<g$, there exists a closed surface $S'$ of strictly positive genus with $2d$ distinct 
paired punctures $p^\pm_1, \dots, p^\pm_d \in S'$ such that the open surface $S\setminus 
(\cup\{a_1,\dots, a_d\})$ is homeomorphic to the surface 
$$
S'':=S' \setminus \{p^+_1,p^-_1, \dots, p^+_d, p^-_d\}\,.
$$
Since $S'$ has strictly positive genus, there exists a continuous closed curve $\gamma'\subset S'' \subset S'$ such that $[\gamma] \not=0 \in H_1(S', \Z)$. Let $\gamma\subset S$ be any smooth curve isotopic to the image of $\gamma' \subset S'$ in $S\setminus 
(\cup\{a_1,\dots, a_d\})$. 

\smallskip
Let us construct a representative of the Poincar\'e dual $P[\gamma] \in H^1(S, \Z)$ 
supported in a compact subset of the open set 
$$S \setminus ( \cup \{a_1,\dots, a_d\} \cup \Sigma_\F)\,.
$$
By (\ref{eq:gamma}) there exists an open tubular neighborhood $ \mathcal U \subset S$ of $\gamma$ such that
\begin{equation}
\label{eq:U}
\overline{\mathcal U} \cap a_i =\emptyset \,, \,\, \text{ \rm for all } j\in \{1, \dots, d\}\,, \quad
\text{\rm and } \overline{\mathcal U} \cap \Sigma_{\F} =\emptyset\,.
\end{equation}
Let $\mathcal V\subset \subset \Cal U$ be an open tubular neighborhood of $\gamma$
in $\Cal U$. Let ${\Cal U}^{\pm}$ be the two connected components of the open set
${\Cal U}\setminus\gamma$ and let ${\Cal V}^{\pm}:={\Cal V}
\cap {\Cal U}^{\pm}$. Let $f:S\to \R$ be a function, 
smooth on $S\setminus\gamma$, with the following properties:
\begin{equation}
 f(p)= 
 \begin{cases} 
 0\,,  \quad \text{ \rm for all } p\in {\Cal U}^- \cup (S \setminus {\Cal U}^+) \,; \\
 1\,, \quad \text{ \rm for all } p\in\overline{ {\Cal V}^+}\,.
 \end{cases} 
\end{equation}
For any $r\in {\Q}\setminus\{0\}$, the $1$-form $\eta_r:=\eta_\F+r df$ is smooth and closed 
(but not cohomologous to $\eta_\F$!) and $\eta_r\to \eta_\F$, as $r\to 0$, in the space of 
smooth $1$-forms on $S$. Since $df=0$ on $S\setminus {\Cal U}$, $\eta_r\equiv\eta_\F$ in a neighbourhood of $\Sigma_\F$, hence, if $r\not=0$ is sufficiently small, $\eta_r(p)=0$ 
if and only if $p\in \Sigma_{\F}$, and the isotopy class of the orientable measured foliation 
${\Cal F}_r:=\{\eta_r=0\}$ belongs to the space $\F(\kappa)$. Since $r\in \Q$, 
the fundamental class $[\eta_r]\in H^1(S, \Sigma_{\Cal F};\Q)$, hence $\F_r$ 
is periodic. The simple closed curves $a_1,\dots,a_d$ are regular leaves of 
$\F_r$. In fact, $df=0$ on $S\setminus {\Cal U}$ and $\cup\{a_1,\dots,a_d\}\subset 
S\setminus \cup\,{\overline {\Cal U}}$, hence $\eta_r\equiv \eta_{\Cal F}$ in a 
neighbourhood of  $\cup\{a_1,\dots,a_d\}$. It follows that $a_1, \dots, a_d$ are 
also regular leaves of the foliation $\F_r$, hence the isotropic subspace 
$\L(\F_r)\subset H_1(S,\R)$ generated by the homology classes of the regular leaves 
of the perturbed foliation $\F_r$ contains the isotropic subspace $\L(\F)$ generated by the 
homology classes of the regular leaves of the foliation $\F$. 

We claim that  $\L(\F_r) \not= \L(\F)$, hence $\text{ \rm dim}\, \L(\F_r) >d:= \text{ \rm dim}\, \L(\F)$. 
In fact, if equality holds $P^{-1}[\eta_r]\in  \L(\F)$ by formula (\ref{eq:Poincaredual}). However, 
$P^{-1} [\eta_\F] \in \L(\F)$, again by formula
(\ref{eq:Poincaredual}), and $P^{-1} [df] = [\gamma] \not\in \L(\F)$ by construction, hence
$$
P^{-1}[\eta_r]=P^{-1}[\eta_{\Cal F}] +  r P^{-1}[df]  \not\in  \L(\F)\,.
$$
The claim is thus proved.

\smallskip
\noindent By a finite iteration of the previous construction, we can show
that the closure in $\F(\kappa)$ of the subset of all Lagrangian 
measured foliations contains the subset of all periodic measured foliations. 
Hence it coincides with the entire space $\F(\kappa)$. The density 
lemma is therefore proved. 
\end{proof}

By the above density lemma, we can derive from our criterion the non-uniform hyperbolicity 
of the Kontsevich--Zorich cocycle for all canonical (absolutely continuous) $\SL$-invariant
measures on connected components of strata of abelian differentials (see \cite{Ftwo}, 
\cite{Avila:Viana}). The resulting proof is a simplified version of the original proof given 
in \cite{Ftwo}, Thm.~8.5.

\begin{corollary} 
\label{cor:canonic}
All canonical measures on connected components of strata of abelian differentials are cuspidal Lagrangian, hence the Kontsevich--Zorich cocycle is non-uniformly hyperbolic with respect to all such measures.
\end{corollary}
\begin{proof} In the coordinates given by the relative period map, all canonical measures are 
(locally) equivalent to the Lebesgue measure and the invariant foliations of the Teichm\"uller
flow are linear. It follows that canonical measures have a local product structure. By Lemma~
\ref{lemma:density} every canonical measure is Lagrangian (on every connected component), 
hence by definition it is cuspidal Lagrangian. By Theorem~\ref{thm:main} the Kontsevich--Zorich 
cocycle is non-uniformly hyperbolic with respect to all such measures.
\end{proof}

\appendix

\section{Other relevant examples \\ by Carlos Matheus}
\label{section:Appendix}

In this Appendix, we present some examples of closed $SL(2,\mathbb{R})$-orbits generated
by square-tiled surfaces which provide interesting examples in the discussion on
G.~Forni's geometric criterion for the non-uniform hyperbolicity of the Kontsevich--Zorich 
cocyle (KZ cocycle for short).

During the preparation of his manuscript about his geometric criterion for the non-vanishing of Lyapunov exponents of the Kontsevich--Zorich cocycle , G.~Forni asked me some natural questions originating from his paper. In particular, the following two questions arose:
\begin{itemize}
\item are there some examples of {\em cuspidal Lagrangian} $SL(2,\mathbb{R})$-invariant ergodic probability measures with {\em non-simple} Lyapunov exponents on the corresponding Kontsevich--Zorich spectrum?
\item are there some examples of {\em cuspidal} $SL(2,\mathbb{R})$-invariant ergodic probability measures whose {\em homological dimension} is {\em strictly less than} the number of positive Lyapunov exponents on the corresponding Kontsevich--Zorich spectrum?
\end{itemize} 

We refer to Definition~\ref{def:LagrangianFol} and Definition~\ref{def:cuspLagragian} of the Introduction for more details on the terms marked in italic (namely, \emph{cuspidal}, \emph{Lagrangian }and \emph{homological dimension}).

The first question is related to Theorem~\ref{thm:main} of the Introduction and the simplicity theorem of Avila and Viana~\cite{Avila:Viana}: in fact, as pointed out by G.~Forni in the Introduction, while his Theorem 2 shows that any {\em cuspidal Lagrangian} $SL(2,\mathbb{R})$-invariant ergodic probability is non-uniformly hyperbolic (i.e., $0$ doesn't belong to the Kontsevich--Zorich spectrum of this measure), it doesn't provide any hints about the simplicity of the KZ cocycle (i.e., the multiplicity of the Lyapunov exponents is $1$). In Section~\ref{s.2covers} of this Appendix, we present certain regular (i.e., unbranched) double-covers of genus $2$ square-tiled surfaces leading to cuspidal Lagrangian 
$SL(2,\mathbb{R})$-invariant ergodic probabilities with multiple (i.e., non-simple) Kontsevich--Zorich spectrum. These examples (together with the ``stairs'' square-tiled cyclic covers mentioned in the Introduction) provide some square-tiled surfaces such that the canonical $SL(2,\mathbb{R})$-invariant ergodic probability measure supported on its $SL(2,\mathbb{R})$-orbit is cuspidal Lagrangian with multiple (non-vanishing) Lyapunov exponents of the KZ cocycle (see Theorem~\ref{t.2covers} below), 
so that the first question has a positive answer.

The second question is related to Theorem~\ref{thm:general} of the the Introduction: while this theorem ensures that any cuspidal $SL(2,\mathbb{R})$-invariant ergodic probability measure with homological dimension $k\in\{1,\dots,g\}$ has $k$ strictly positive Lyapunov exponents in its Kontsevich--Zorich spectrum \emph{at least}, and the lower bound on the number of non-vanishing Kontsevich--Zorich exponents provided by this result is the best possible in view of the {\em maximally degenerate examples}~\cite{ForniSurvey},~\cite{Forni:Matheus}  of cuspidal $SL(2,\mathbb{R})$-invariant ergodic probabilities with homological dimension $1$ and {\em exactly} one non-vanishing Kontsevich--Zorich exponent (see also \cite{Forni:Matheus:Zorich:one}), it doesn't give upper bounds on the number of non-vanishing Kontsevich--Zorich exponents based on the homological dimension. In Section~\ref{s.jc} of this Appendix, as it was suggested by G.~Forni during our conversations, we show that a family of square-tiled cyclic covers indexed by odd integers $q\geq 3$ studied by J.-C.~Yoccoz and myself (see~\cite{Matheus:Yoccoz}, $\S$ 3.1) give cuspidal $SL(2,\mathbb{R})$-invariant ergodic probabilities with homological dimension 1 such that the number of non-vanishing Kontsevich--Zorich exponents equal to $1+(q-3)/2$ when $q=3$ (mod $4$) and $1+(q-1)/2$ when $q=1$ (mod $4$), so that the answer to the second question is also positive (see Theorem~\ref{t.jc} below).

Closing this introduction, I would like to acknowledge G.~Forni for his kind invitation to contribute with this Appendix, A.~Zorich for allowing me to use his excellent computer programs in order to numerically test some ideas and conjectures, and G.~Forni and A.~Zorich  for several fruitful discussions.

\subsection{Cuspidal Lagrangian measures with multiple spectrum}
\label{s.2covers}

Given an Abelian differential $\omega_S$ on a Riemann surface $S$, one can produce further examples of Abelian differentials on Riemann surfaces by the following {\em covering} procedure: given a (possibly ramified) covering $p:R\to S$ of Riemann surfaces, we can define an Abelian differential $\omega_R$ by pulling back $\omega_S$ under the covering map $p$, i.e., $\omega_R=p^*(\omega_S)$. In this situation, we can relate the Kontsevich--Zorich spectrum $\mathcal{L}(S,\omega_S)$ of $(S,\omega_S)$ with the Kontsevich--Zorich spectrum $\mathcal{L}(R,\omega_R)$ of $(R,\omega_R)$:

\begin{lemma}
Let $p:R\to S$ be a possibly ramified covering. Then, for any Abelian differential $\omega_S$ on $S$, we have the inclusion $\mathcal{L}(S,\omega_S)\subset\mathcal{L}(R,\omega_R)$, where $\omega_R:=p^*(\omega_S)$. In other words, any Kontsevich--Zorich exponent of $(S,\omega_S)$ is also a Kontsevich--Zorich exponent of $(R,p^*(\omega_S))$.
\end{lemma}

\begin{remark} While this elementary lemma is a common knowledge of several authors (for instance, 
G.~Forni and A.~Zorich were aware of it for quite a long time), I included a brief indication of its proof for sake of completeness. 
\end{remark}

\begin{proof} A direct inspection of the definitions (see Section 2 of the main article) shows that the covering $p:R\to S$ induces a natural injective linear isometry 
$$(H^1(S,\mathbb{R}),\|.\|_{\omega_S})\to (H^1(R,\mathbb{R}),\|.\|_{\omega_R}),$$
where $\|.\|_{\omega}$ stands for the {\em Hodge norm} on $H^1(M,\mathbb{R})$. The desired lemma follows from Lemma 2.1' of~\cite{Ftwo} (or equivalently, Lemma 4.3 of~\cite{ForniSurvey}).
\end{proof}

In the sequel, we'll specialize this covering procedure to the case of {\em double unramified} covers $p:R\to S$ of a {\em square-tiled surface}\footnote{Recall that, in general, $(S,\omega_S)$ is square-tiled surface when the Abelian differential $\omega_S$ has {\em integral} periods. Alternatively (and equivalently), we say that $(S,\omega_S)$ is a square-tiled surface when its stabilizer $SL(S,\omega_S)$ (called {\em Veech group}) under the natural $SL(2,\mathbb{R})$ action on the moduli space of Abelian differentials is commensurable to $SL(2,\mathbb{Z})$. For more details on square-tiled surfaces, see, e.g.,~\cite{Zorich6} and references therein.} $(S,\omega_S)$ of genus 2. In this situation, $(R,\omega_R:=p^*(\omega_S))$ is a square-tiled surface of genus 3. Of course, since we're dealing with unbranched coverings, if $(S,\omega_S)\in H(2)$ (i.e., $\omega_S$ has one double zero), then $(R,\omega_R)\in H(2,2)$ (i.e., $\omega_R$ has two double zeroes), and if $(S,\omega_S)\in H(1,1)$ (i.e., $\omega_S$ has two simple zeroes), then $(R,\omega_R)\in H(1,1,1,1)$ (i.e., $\omega_R$ has four simple zeroes). On the other hand, after the works of M.~Bainbridge~\cite{Bainbridge} (see also~\cite{Eskin:Kontsevich:Zorich}), we know that the Kontsevich--Zorich spectrum $\mathcal{L}(S,\omega_S)$ in the case of a genus $2$ surface $S$ is:
\begin{equation*}
\mathcal{L}(S,\omega_S)=\left\{
\begin{array}{rl}
\{1,1/3\} & \textrm{if } (S,\omega_S)\in H(2), \\
\{1,1/2\} & \textrm{if } (S,\omega_S)\in H(1,1).
\end{array}\right.
\end{equation*} 
Thus, the preceding lemma implies the following fact:

\begin{corollary}\label{c.2covers}
Let $p:R\to S$ be an unramified double covering
of a genus two square-tilled surface $(S,\omega_S)$.
Then, 
$$ 
\begin{array}{rl}
\mathcal{L}(R,p^*\omega_S) \supset\{1, 1/3\} & \textrm{if } (R,p^*\omega_S)\in H(2,2), \\ 
\mathcal{L}(R,p^*\omega_S) \supset\{1, 1/2\} & \textrm{if } (R,p^*\omega_S)\in H(1,1,1,1).
\end{array}
$$ 
\end{corollary}

In order to simplify the exposition, we'll present our square-tiled surfaces in a combinatorial fashion, namely, we label its unit squares (tiles) using positive integers $i=1,\dots,N$ and we consider a pair of permutations $(h,v)\in S_N\times S_N$ such that $h(i)$ (resp. $v(i)$) is the neighbor to the right (resp. on the top) of the square $i$. Since our Riemann surfaces are connected, we require that $h$ and $v$ act transitively on $\{1,\dots,N\}$. Also, since we can relabel the squares (tiles) of our surface without changing it, we'll say that $(h_0,v_0)$ is equivalent to $(h_1,v_1)$ whenever they are {\em simultaneously} conjugated (i.e., there exists $\phi\in S_N$ such that $h_1=\phi^{-1}h_0\phi$ {\em and} $v_1=\phi^{-1}v_0\phi$). Below, we will always write permutations through their cycles (and we'll write even their 1-cycles to make the total number of square tiles of our surfaces more evident). 

Consider the square-tiled surface associated to $h_{S_0}=(1,2)(3)$ and $v_{S_0}=(1,3)(2)$. It is a L-shaped square-tiled surface $S_0$ formed by 3 unit squares glued via the recipe provided by $h_{S_0}$ and $v_{S_0}$. We can form double covers $R_0$ of $S_0$ by taking two copies of $S_0$ and changing the side identifications conveniently (using a pair of permutations $h_{R_0}$ and $v_{R_0}$). For our purposes, we take $h_{R_0}=(1,2,3,4)(5,6)$ and $v_{R_0}=(1,5)(2)(3,6)(4)$.  

\begin{theorem}\label{t.2covers} The canonical (absolutely continuous) $SL(2,\mathbb{R})$-invariant ergodic probability measure $\mu_{R_0}$ supported on the (closed) $SL(2,\mathbb{R})$-orbit of the square-tiled surface $(R_0,\omega_{R_0})$ associated to 
$$h_{R_0}=(1,2,3,4)(5,6) \quad \text{\rm and} \quad v_{R_0}=(1,5)(2)(3,6)(4)$$ 
is (cuspidal) Lagrangian and its Kontsevich--Zorich spectrum is multiple:
$$\mathcal{L}(R_0,\omega_{R_0})=\{1, 1/3,1/3\}.$$
\end{theorem}

\begin{proof} We begin by showing that $\mu_{R_0}$ is Lagrangian. The vertical foliation of $R_0$ has 4 cylinders $C_{(1,5)}$, $C_{(2)}$, $C_{(3,6)}$ and $C_{(4)}$ (where the subindices is composed of the labellings of all squares forming the corresponding cylinders). Denoting by $\gamma_{(1,5)}$, $\gamma_{(2)}$, $\gamma_{(3,6)}$, $\gamma_{(4)}$ the homology classes of the waist curves of these cylinders, it is easy to see that they generate a 3-dimensional subspace of $H_1(R_0,\mathbb{R})$ (because $\gamma_{(1,5)}=\gamma_{(3,6)}$ and $\gamma_{(1,5)}$, $\gamma_{(2)}$, $\gamma_{(4)}$ are linearly independent by direct calculation). Since the genus of $R$ is 3, we're done.

Next, we compute the Kontsevich--Zorich spectrum of $(R_0,\omega_{R_0})$. We'll accomplish this task with the aid of the following formula of A.~Eskin, M.~Kontsevich and A.~Zorich~\cite{Eskin:Kontsevich:Zorich} for the {\em sum} of Kontsevich--Zorich exponents associated to square-tiled surfaces:
\begin{equation}
\begin{aligned}
\lambda_1+\dots+\lambda_g&=\frac{1}{12}\sum\limits_{i=1}^n \frac{m_i(m_i+2)}{m_i+1} \\
&+\frac{1}{\#SL(2,\mathbb{Z})\cdot P_0}\sum\limits_{P_i\in SL(2,\mathbb{Z})\cdot P_0} \sum\limits_{P_i=\cup cyl_{ij}}\frac{h_{ij}}{w_{ij}}.
\end{aligned}
\end{equation}
Here, $\lambda_1,\dots,\lambda_g$ are the Kontsevich--Zorich exponents of a square-tiled surface $P_0$ of genus $g$ belonging to the stratum $H(m_1,\dots,m_n)$. Also, $SL(2,\mathbb{Z})\cdot P_0$ denotes the (finite) orbit of $P_0$ under the action of $SL(2,\mathbb{Z})$. For each $P_i\in SL(2,\mathbb{Z})\cdot P_0$, the decomposition of $P_i$ into {\em maximal} horizontal cylinders is denoted by $P_i=\cup cyl_{ij}$. Furthermore, $h_{ij}$ denotes the height of the cylinder $cyl_{ij}$ and $w_{ij}$ is the length of the waist curve of the cylinder $cyl_{ij}$. 

In the case of the genus $3$ square-tiled surface $(R_0,\omega_{R_0})\in H(2,2)$, we combine this formula for the sum of Lyapunov exponents together with our knowledge of two Lyapunov exponents of $(R_0,\omega_0)$ (from Corollary~\ref{c.2covers}) to get: 
\begin{equation}\label{e.2covers-1}
1+\lambda_2(R_0)+\frac{1}{3} = \frac{4}{9} + \frac{1}{\#SL(2,\mathbb{Z})\cdot R_0}\sum\limits_{R_i\in SL(2,\mathbb{Z})\cdot R_0} \sum\limits_{R_i=\cup cyl_{ij}}\frac{h_{ij}}{w_{ij}},
\end{equation}
where $\lambda_2(R_0)$ is the second Kontsevich--Zorich exponent of $(R_0,\omega_{R_0})$. This reduces our task to the computation of $SL(2,\mathbb{Z})\cdot R_0$. Keeping this goal in mind, we'll work with the generators $T=\left(\begin{array}{cc} 1 & 1 \\ 0 & 1\end{array}\right)$ and $J=\left(\begin{array}{cc} 0 & -1 \\ 1 & 0\end{array}\right)$ of $SL(2,\mathbb{Z})$. Their actions on a square-tiled surface presented as a pair of permutations $(h,v)$ is given by the Nielsen transformations $T(h,v)=(h,vh^{-1})$ and $J(h,v)=(v^{-1},h)$. Put
\begin{itemize}
\item $h_{R_1}:=h_{R_2}:=h_{R_3}:=h_{R_0}=(1,2,3,4)(5,6)$; 
\item $v_{R_1}:=(1,4,6)(2,5,3)$, $v_{R_2}:=(1,6,3,5)(2,4)$, \\ $v_{R_3}:=(1,2,6)(3,4,5)$;
\item $h_{R_4}:=h_{R_5}:=(1,5)(2)(3,6)(4)$;
\item $v_{R_4}:=h_{R_0}$, $v_{R_5}:=(1,6,4)(2,3,5)$;
\item $h_{R_6}:=h_{R_7}:=h_{R_8}:=(1,6,4)(2,3,5)$;
\item $v_{R_6}:=h_{R_0}$, $v_{R_7}:=(1)(2,6)(3)(4,5)$, \\ $v_{R_8}:=(1,5,3,6)(2,4)$.
\end{itemize}
Let $R_i$ be the square-tiled surface associated to $(h_{R_i},v_{R_i})$, $i=0,\dots,8$. A straightforward calculation shows that $SL(2,\mathbb{Z})\cdot R_0=\{R_0,\dots, R_8\}$ and it is organized as follows:
\begin{itemize}
\item $T(R_0)=R_1$, $T(R_1)=R_2$, $T(R_2)=R_3$, $T(R_3)=R_0$;
\item $J(R_0)=R_4$ and $J(R_3)=R_0$;
\item $T(R_4)=R_5$ and $T(R_5)=R_4$;
\item $J(R_1)=R_6$ and $J(R_6)=R_1$;
\item $T(R_6)=R_7$, $T(R_7)=R_8$, $T(R_8)=R_6$;
\item $J(R_7)=R_5$ and $J(R_5)=R_7$;
\item $J(R_8)=R_4$ and $J(R_4)=R_8$;
\item $J(R_2)=R_2$.
\end{itemize}
In the literature, each $T$-orbit is called a {\em cusp}. The above description says that $SL(2,\mathbb{Z})\cdot R_0=\mathcal{C}_1\cup\mathcal{C}_2\cup \mathcal{C}_3$ is the disjoint union of 3 cusps, namely, $\mathcal{C}_1=\{R_0,R_1,R_2,R_3\}$, $\mathcal{C}_2=\{R_4,R_5\}$ and $\mathcal{C}_3=\{R_6,R_7,R_8\}$. The contribution of two square-tiled surfaces belonging to a fixed cusp to the sum appearing in the right-hand side of~\eqref{e.2covers-1} are equal (because their horizontal permutations are the same). Therefore, it suffices to compute this contribution at an arbitrarily chosen surface inside a fixed cusp, multiply it by the length (size) of this cusp and then sum up over all cusps in order to determine the sum of the right-hand side of~\eqref{e.2covers-1}. In the case at hand, we have:
\begin{itemize}
\item each surface in the $1^{st}$ cusp $\mathcal{C}_1$ contributes with $\frac{1}{4}+\frac{1}{2}=\frac{3}{4}$ and $\#\mathcal{C}_1=4$;
\item each surface in the $2^{nd}$ cusp $\mathcal{C}_2$ contributes with $2\cdot\frac{1}{2}+2\cdot1=3$ and $\#\mathcal{C}_2=2$;
\item each surface in the $3^{rd}$ cusp $\mathcal{C_3}$ contributes with $\frac{1}{3}+\frac{1}{3}=\frac{2}{3}$ and $\mathcal{C}_3=3$.
\end{itemize} 
Plugging this information into~\eqref{e.2covers-1} (and noting that $\#SL(2,\mathbb{Z})=9$), we obtain:
\begin{equation*}
\frac{4}{3}+\lambda_2(R_0) = \frac{4}{9}+\frac{1}{9}\left(4\cdot\frac{3}{4}+2\cdot3+3\cdot\frac{2}{3}\right) = \frac{5}{3},
\end{equation*}
i.e., $\lambda_2(R_0)=1/3$ and, {\em a fortiori}, $\mathcal{L}(R_0,\omega_{R_0})=\{1, 1/3,1/3\}$.
\end{proof}

\subsection{Non-zero exponents beyond the homological dimension} \label{s.jc}
In $\S3.1$ of~\cite{Matheus:Yoccoz}, the authors introduced a family of closed $SL(2,\mathbb{R})$-orbits, indexed by odd integers $q\geq 3$, of  square-tiled surfaces $(M_q,\omega_q)$  isomorphic 
to the desigularization of the algebraic curve
$$w^{2q}=z^{q-2}(z^2-1)$$ equipped with the (unit area) Abelian differential 
$$\omega_q=\frac{2\sqrt{2\pi}}{(\Gamma(1/4))^2}\frac{z^{\frac{q-3}{2}}dz}{w^q}.$$
For a pictorial description of these square-tiled surfaces (together with the zeroes of $\omega_q$), see Figure 3 (for the general case), Figure 4 (for $q=3$) and Figure 5 (for $q=5$) of~\cite{Matheus:Yoccoz}. As pointed out in this article, $(M_3,\omega_3)$ corresponds to the totally degenerate genus $4$ example of~\cite{Forni:Matheus} (see also \cite{Forni:Matheus:Zorich:one}), so that this family maybe considered as a natural generalization of the exceptionally symmetric case $q=3$. Generally speaking, $(M_q,\omega_q)\in H(q-1,q-1,q-1)$ (i.e., $\omega_q$ has 3 zeroes of order $(q-1)$, and, a fortiori, the genus of $M_q$ is $(3q-1)/2$). However, the similarities between the cases $q=3$ and $q\geq 5$ end here: for instance, it was proved in~\cite{Matheus:Yoccoz} that the Veech group of $(M_3,\omega_3)$ is $SL(2,\mathbb{Z})$, while the Veech group of $(M_q,\omega_q)$ is 
\begin{equation}\label{e.Veech-jc}
SL(M_q,\omega_q)=\{M\in SL(2,\mathbb{Z}): M\equiv I \textrm{ or } M\equiv J \, \textrm{ mod }2\}
\end{equation}
for every (odd) $q\geq 5$. Also, the Kontsevich--Zorich spectrum of $(M_3,\omega_3)$ is maximally degenerate (i.e., all but one of the Kontsevich--Zorich exponents are zero), while $(M_q,\omega_q)$ are never totally degenerate when $q\geq 5$ (see Remark 3.2 of~\cite{Matheus:Yoccoz}). 

The next result shows that the non-maximally degenerate examples associated to $(M_q,\omega_q)$ ($q\geq 5$) have homological dimension equal to $1$:

\begin{theorem}\label{t.jc} For every odd $q\geq 5$, the canonical (absolutely continuous) $SL(2,\mathbb{R})$-invariant ergodic probability $\mu_q$ supported on the closed orbit $SL(2,\mathbb{R})\cdot (M_q,\omega_q)$ has homological dimension equal to $1$ and the number of strictly positive Kontsevich--Zorich exponents is:
\begin{itemize}
\item $1+(q-1)/4$ when $q\equiv 1$ (mod $4$) or
\item $1+(q-3)/4$ when $q\equiv 3$ (mod $4$).
\end{itemize}
\end{theorem}

\begin{proof}Fix $q\geq 5$ odd. The description~\eqref{e.Veech-jc} of the Veech group of $(M_q,\omega_q)$ shows that it is a index $3$ subgroup of $SL(2,\mathbb{Z})$ and 
$$T=\left(\begin{array}{cc} 1 & 1 \\ 0 & 1 \end{array}\right) \,, \quad S=\left(\begin{array}{cc} 1 & 0 \\ 1 & 1 \end{array}\right)\,,\quad ST$$ 
are distinct representatives of the cosets of $SL(2,\mathbb{Z})/SL(M_q,\omega_q)$. Therefore, we can compute the homological dimension of $\mu_q$ by calculating the dimensions of the subspaces of $H_1(M_q,\omega_q)$ generated by the homology classes of the waist curves of cylinders in the horizontal and main diagonal (i.e., slope $1$) directions. Using the notations of Section $\S3.1$ of~\cite{Matheus:Yoccoz}, we have that the waist curves of the two cylinders of $(M_q,\omega_q)$ along the horizontal direction are homologous to $\sum\limits_{i\in\mathbb{Z}/q}(\sigma_i+\sigma_i'):=\sigma$, and the waist curves of the two cylinders of $(M_q,\omega_q)$ along the main diagonal direction are homologous to $\sum\limits_{i\in\mathbb{Z}/q}(\sigma_i+\sigma_i'+\zeta_{i-1}+\zeta_{i+1}'):=\sigma+\zeta$. It follows that the homological dimension of $\mu_q$ is equal to $1$.

The computation of the number of positive Kontsevich--Zorich exponents can be done with the aid of the results of~\cite{Eskin:Kontsevich:Zorich:cyclic}: in fact, $(M_q,\omega_q)$ is a square-tiled cyclic cover of type $M_{2q}(1,1,q-1,q-1)$ (in the notation of~\cite{Forni:Matheus:Zorich:one}), so that we can use the (action $T^*$ of the) automorphism $T(z,w)=(z,\zeta w)$, $\zeta$ a primitive $2q$-th root of unit, of the algebraic curve $M_q$ to decompose the complex cohomology $H^1(M_q,\mathbb{C})$ into a direct sum of eigenspaces $V_k=Ker(T^*-\zeta^k \textrm{Id})$, $0<k<2q$. Using this decomposition, one can determine the number of positive Kontsevich--Zorich exponents:  the number of positive Kontsevich--Zorich exponents is 
\begin{equation}\label{e.FMZ}
\#\{0<k<2q: \textrm{dim}_{\mathbb{C}} V_k^{1,0} = \textrm{dim}_{\mathbb{C}} V_{N-k}^{1,0}=1\}
\end{equation}
where $V^{1,0}$ denotes the $(1,0)$-part of $V\subset H^1(M_q,\mathbb{C})$. On the other hand, for a general cyclic cover of type $M_N(a_1,a_2,a_3,a_4)$, we have that 
$$\textrm{dim}_{\mathbb{C}} V_{N-k}^{1,0}=\sum\limits_{\mu=1}^4\left\{\frac{ka_{\mu}}{N}\right\}-1$$
where $\{x\}$ is the fractional part of $x$. See, e.g.,~\cite{Eskin:Kontsevich:Zorich:cyclic} for more details. In the particular case of $(M_q,\omega_q)$, i.e., $N=2q$, $a_1=a_2=1$ and $a_3=a_4=q-1$, we get 
$$\textrm{dim}_{\mathbb{C}} V_{2q-k}^{1,0}=2\left\{\frac{k}{2q}\right\}+2\left\{\frac{(q-1)k}{2q}\right\}-1.$$
In particular,
\begin{itemize}
\item $\textrm{dim}_{\mathbb{C}} V_{2q-k}^{1,0}=0$ when $0<k<q$ and $k$ {\em even};
\item $\textrm{dim}_{\mathbb{C}} V_{2q-k}^{1,0}=1$ when $q<k<2q$ and $k$ {\em even};
\item $\textrm{dim}_{\mathbb{C}} V_{2q-k}^{1,0}=0$ when $0<k<q/2$ and $k$ {\em odd};
\item $\textrm{dim}_{\mathbb{C}} V_{2q-k}^{1,0}=2$ when $3q/2<k<2q$ and $k$ {\em odd};
\item $\textrm{dim}_{\mathbb{C}} V_{2q-k}^{1,0}=1$ when $q/2<k<3q/2$ and $k$ {\em odd};
\end{itemize}
Inserting this information into formula~\eqref{e.FMZ}, we see that the number of positive Kontsevich--Zorich exponents of $(M_q,\omega_q)$ is
$$\#\left\{\frac{q}{2}<k<\frac{3q}{2}: k \, \textrm{odd}\right\} =\left\{\begin{array}{rl}1+ (q-1)/4& \textrm{ if } q\equiv 1 (\textrm{mod }4)\\ 1+ (q-3)/4 & \textrm{ if } q\equiv 3 (\textrm{mod } 4)\end{array}\right.$$
This completes the argument.
\end{proof}

\begin{remark}
An alternative way of counting the number of positive exponents of $(M_q,\omega_q)$ uses Remark 3.2 of~\cite{Matheus:Yoccoz}. Using the notation of that paper, we have that the two-dimensional tautological subspace $H_1^{st}$ contributes with the exponent $+1$ and the other exponents come from $H_{\tau}\oplus\breve{H}$. The $(q-1)$-dimensional subspace $H_\tau$ contains only zero exponents (since the action of the affine group is through a finite group) and $\breve{H}=\bigoplus\limits_{j=1}^{q-1} \breve{H}(\rho^j)$ where $\breve{H}(\rho^j)$ are 2-dimensional invariant subspaces (under the action of the affine group) naturally indexed by the powers $\rho^j$ of the primitive $q$-th root of unit $\rho:=\exp(2\pi i/q)$. The trace of the actions of the generators $\widetilde{S}^2$, $\widetilde{T}^2$ of $\textrm{Aff}_{(1)}(M_q,\omega_q)$ (the subgroup of the affine group fixing the zeroes of $\omega_q$) as well as the trace of $\widetilde{S}^2\widetilde{T}^2$ on $\breve{H}(\rho^j)$ were calculated in~\cite{Matheus:Yoccoz}, $\S3.2$. The outcome of this calculation is the fact that these traces are $<2$ when $q/4<j<q$, while the trace of the action of $\widetilde{S}^2\widetilde{T}^2$ is $>2$ when $1\leq j < q/4$. This implies that the affine group of $(M_q,\omega_q)$ acts on $\breve{H}(\rho^j)$ by a compact subgroup (of elliptic matrices) when $q/4<j<q$, while it acts non-trivially (with some hyperbolic elements) on $\breve{H}(\rho^j)$ when $1\leq j<q/4$. Consequently, the quantity of positive exponents is 
$$1+\#\{j: 1\leq j<q/4\}=\left\{\begin{array}{rl}1+ (q-1)/4& \textrm{ if } q\equiv 1 (\textrm{mod }4)\\ 1+ (q-3)/4 & \textrm{ if } q\equiv 3 (\textrm{mod } 4)\end{array}\right.$$
\end{remark}

\begin{remark}
By the recent work of A. Eskin, M.~Kontsevich and A.~Zorich (see \cite{Eskin:Kontsevich:Zorich:cyclic}), 
it is possible to determine the precise value of the Kontsevich--Zorich exponents of $(M_q,\omega_q)$: indeed, going back to the notation of the proof of the theorem above, the main result of \cite{Eskin:Kontsevich:Zorich:cyclic} implies that every $0<k<2q$ with $\textrm{dim}_{\mathbb{C}} V_k^{1,0} = \textrm{dim}_{\mathbb{C}} V_{2q-k}^{1,0}=1$ gives rise to a (positive) Kontsevich--Zorich exponent 
$$\lambda(k)=2\min\{\{k/2q\},1-\{k/2q\}\}.$$
Taking into account the symmetry $\lambda(k)=\lambda(2q-k)$, this means that positive Kontsevich--Zorich exponents of $(M_q,\omega_q)$ are  
$$\{k/q:0<k<q,\,\textrm{dim}_{\mathbb{C}}V_k^{1,0}=\textrm{dim}_{\mathbb{C}}V_{2q-k}^{1,0}=1\}$$
where each value $k/q$ appears with multiplicity 2.
\end{remark}


\begin{thebibliography}{9}

\bibitem[ABEM]{Athreya:Bufetov:Eskin:Mirzakhani} J.~Athreya, A.~Bufetov, A.~Eskin and M.~Mirzakhani,
Lattice {P}oint {A}symptotics and {V}olume {G}rowth on {T}eichm\"uller space, 
\emph{arXiv:math/0610715v3}, 2010, 1--39.

\bibitem[AtF]{Athreya:Forni} J.~Athreya and G.~Forni, \emph{Deviation of ergodic averages for rational polygonal billiards}, Duke Math. J. \textbf{144:2}, 2008, 285-319. 

\bibitem[AvF]{Avila:Forni} A.~Avila and G.~Forni,  \emph{Weak mixing for interval exchange transformations and translation flows}, Ann. of Math. \textbf{165:2}, 2007, 637--664.

\bibitem[AvV]{Avila:Viana} A.~Avila and M.~Viana, \emph{Simplicity of
{L}yapunov Spectra: Proof of the Kontsevich--Zorich conjecture},
Acta Math., \textbf{198}, 2007, 1--56.

\bibitem[Ba]{Bainbridge} M.~Bainbridge, \emph{Euler characteristics
of Teichm\"uller curves in genus two}, Geometry \& Topology, \textbf{11},
2007, 1887--2073.

\bibitem[BaMo]{Bainbridge:Moeller} M.~Bainbridge and M.~M\"oller, 
\emph{The Deligne-Mumford compactification of the
real multiplication locus and Teichm¬uller curves
in genus three},  arXiv:0911.4677v1, 2009, 1--80
(to appear on Acta Math.)

\bibitem[BMo]{Bouw:Moeller} I.~Bouw and M.~M\"oller,
\emph{Teichm\"uller curves, triangle groups and Lyapunov exponents},
Ann. of Math., \textbf{172}, 2010, 139--185.

\bibitem[Bu1]{Bufetov1} A.~Bufetov, \emph{Finitely-additive measures on the asymptotic foliations of a Markov compactum},  arXiv:0902.3303v1, 2009, 1--29.

\bibitem[Bu2]{Bufetov2} \bysame, \emph{Limit Theorems for Translation Flows}, arXiv:0804.3970v3, 2010, 1--69.

\bibitem[Ca]{Calta} K.~Calta, \emph{Veech surfaces and complete periodicity in genus two}, 
J. Amer. Math. Soc. \textbf{17:4}, 2004, 871--908.


\bibitem[EKZ1]{Eskin:Kontsevich:Zorich:cyclic}A.~Eskin, M.~Kontsevich and
A.~Zorich, \emph{Lyapunov spectrum of square-tiled cyclic covers}, to appear on Journal of
Modern Dynamics.

\bibitem[EKZ2]{Eskin:Kontsevich:Zorich} \bysame, \emph{Sum of Lyapunov exponents of the Hodge bundle with respect to the Teichm\"uller geodesic flow}, preprint, 2010.

\bibitem[EM]{Eskin:Mirzakhani} A.~Eskin and M.~Mirzakhani, 
\emph{On invariant and stationary measures for the  $\SL$ action on moduli space},
preprint, 2010.

\bibitem[FK]{FK} H.~M.~Farkas and I.~Kra, \emph{Riemann Surfaces},
Springer-Verlag, New York, 1992 (second edition).

\bibitem[Fa]{Fay} J.~D.~Fay, \emph{Theta Functions on Riemann Surfaces}, Springer-Verlag 
LNM \textbf{352}, Heidelberg, 1973.

\bibitem[F1]{Fone} G.~Forni, \emph{Solutions of the cohomological equation
for area-preserving flows on compact surfaces of higher genus}, Ann. of Math.
\textbf{146}, 1997, 295--344.

\bibitem[F2]{Ftwo} \bysame, \emph{Deviation of ergodic averages
for area-preserving flows on surfaces of higher genus}, Ann. of Math.
{\bf 155}, 2002, 1--103.

\bibitem[F3]{ForniSurvey} \bysame, \emph{On the Lyapunov exponents of
the Kontsevich--Zorich cocycle}, Handbook of Dynamical Systems v. 1B,
B.~Hasselblatt and A.~Katok, eds., Elsevier, 2006, 549--580.

\bibitem[FM]{Forni:Matheus}  G.~Forni and  C.~Matheus, \emph{An
example of a Teichmuller disk in genus $4$ with degenerate
Kontsevich--Zorich spectrum}, arXiv:0810.0023, 2008, 1--8.

\bibitem[FMZ1]{Forni:Matheus:Zorich:one}  G.~Forni, C.~Matheus and A.~Zorich,
\emph{Square-tiled cyclic covers}, to appear on Journal of Modern Dynamics.

\bibitem[FMZ2]{Forni:Matheus:Zorich:two}  \bysame, \emph{Lyapunov spectrum of equivariant subbundles of the Hodge bundle}, preprint, 2010.

\bibitem[HeS]{Herrlich:Schmithuesen} F.~Herrlich and G.~Schmith\"usen, \emph{
An extraordinary origami curve}, Math. Nachr.  \textbf{281:2}, 2008, 219--237.

\bibitem[HuLa]{Hubert:Lanneau} P.~Hubert and E.~Lanneau, \emph{Veech groups without 
parabolic elements}, Duke Math. J. \textbf{133:2}, 2006, 335--346.


\bibitem[HuS]{Hubert:Schmidt} P.~Hubert and T.~A.~Schmidt, \emph{Invariants of translation surfaces}, 
Ann. Inst. Fourier, Grenoble \textbf{51:2}, 2001, 461--495.

\bibitem[KS]{Kenyon:Smillie} R.~Kenyon and J.~Smillie, \emph{Billiards in rational-angled triangles}, Comment. Math. Helv. 75, (2000) 65--108.

\bibitem[K]{Kontsevich} M.~Kontsevich, \emph{Lyapunov exponents and
{H}odge theory}, in ``The mathematical beauty of physics'', Saclay,
1996. Adv. Ser. Math. Phys. \textbf{24}. World Scientific, River
Edge, NJ, 1997, 318--332.

\bibitem[KZ]{Kontsevich:Zorich} M.~Kontsevich and A.~Zorich,
\emph{Connected components of the moduli
spaces  of  Abelian differentials}, Invent. Math.,
\textbf{153:3}, 2003, 631--678.

\bibitem[Kr]{Krikorian} R.~Krikorian, \emph{D\'eviations de moyennes ergodiques, flots de Teichm\"uller et cocycle de Kontsevich--Zorich [d'apr\`es Forni, Kontsevich, Zorich, \dots]}, S\'eminaire Bourbaki, Volume 2003/2004,  n. 927, Novembre 2003,
59--93.

\bibitem[Ma1]{Masur1}  H.~Masur, \emph{On a Class of Geodesics in Teichmuller Space}, 
Ann. of Math., \textbf{102}, 1975, 205--221.

\bibitem[Ma2]{Masur2} \bysame, \emph{Interval exchange transformations
and measured foliations}, Ann. of Math., \textbf{115}, 1982,
169--200.


\bibitem[MY]{Matheus:Yoccoz} C.~Matheus and J.-C.~Yoccoz, \emph{The action of the affine 
diffeomorphisms on the relative homology group of certain exceptionally symmetric 
origamis},  Journal of Modern Dynamics, \textbf{4}, n. 3, 2010, 453--486.

\bibitem[Mc1]{McMullen1} C.~McMullen, \emph{Dynamics of $SL_2(\R)$ over moduli space in genus two}, Ann. of Math.  \textbf{165}, 2007, 397--456.

\bibitem[Mc2]{McMullen2} \bysame, \emph{Prym varieties and Teichm\"uller curves}, 
Duke Math. J. \textbf{133:3}, 2006, 569--590.

\bibitem[Mc3]{McMullen3} \bysame, \emph{Braid groups and Hodge theory}, 
preprint, p. 1-63 (to appear on Math. Ann.).

\bibitem[M\"o1]{Moeller1} M.~M\"oller, \emph{Periodic points on Veech surfaces and the 
{M}ordell--{W}eil group over a {T}eichm\"uller curve}, Invent. Math. \textbf{165:3}, 2006, 633--649.

\bibitem[M\"o2]{Moeller2} \bysame, \emph{{S}himura- and {T}eichm\"uller curves}, 
arXiv:math/0501333v2, 2010, 1--28.

\bibitem[NS]{Nevo:Stein} A.~Nevo and E.~M.~Stein, \emph{{A}nalogs of {W}iener's ergodic theorems for semisimple groups. I},  Ann. of Math. \textbf{145}, No. 3, 1997, 565--595.

\bibitem[Rau]{Rauzy} G.~Rauzy, \emph{Echanges d'intervalles et transformations induites}, 
Acta Arith. \textbf{34}, 1979, 315--328.

\bibitem[SW]{Smillie:Weiss} J.~Smillie and B.~Weiss,  \emph{Minimal sets for flows on moduli space}
Israel J. Math. \textbf{142}, 2004, 249--260.

\bibitem[Tr]{Trevino} R.~Trevi\~no, \emph{On the non-uniform hyperbolicity of the Kontsevich--Zorich
cocycle for quadratic differentials}, preprint, 2010 (arXiv:1010.1038v2).

\bibitem[V1]{Veech1} W.~Veech, \emph{Gauss measures for transformations on the space of interval exchange maps}, Ann. of Math.,   \textbf{115}, 1982,  201--242.

\bibitem[V2]{Veech2} \bysame, \emph{The Teichm\"uller Geodesic Flow}, Ann.
of Math., \textbf{124}, 1986, 441--530.

\bibitem[V3]{Veech3} \bysame, \emph{Teichm\"uller curves in moduli space, Eisenstein series and an application to triangular billiards}, Invent. Math., \textbf{97:3},  553--583, 1989.

\bibitem[V4]{VeechBrin} \bysame, \emph{The Forni cocycle},  Journal of Modern Dynamics
\textbf{2}, No. 3, 2008, 375--395.

\bibitem[Ya]{Yamada} A.~Yamada, \emph{Precise variational formulas for abelian differentials},
Kodai Math. J. \textbf{3}, 1980, 114--143.

\bibitem[Zo1]{Zorich1}  A.~Zorich, \emph{Asymptotic flag of an orientable measured foliation
on a surface}, in  {\it Geometric Study of Foliations}, World Scientific, 	
1994, 479--498.

\bibitem[Zo2]{Zorich2}, \bysame \emph{Finite Gauss measure on the space of interval exchange 
transformations. Lyapunov exponents}, Ann. Inst. Fourier (Grenoble) \textbf {46}, 
1996, 325--370.

\bibitem[Zo3]{Zorich3} 
\bysame, \emph{Deviation for interval exchange transformations},
Ergod. Th. \& Dynam. Sys. \textbf{17}, 1997, 1477--1499.

\bibitem[Zo4]{Zorich4} 
\bysame, \emph{On hyperplane sections of periodic surfaces},
in {\it Solitons, Geometry and Topology: on the Crossroad }(V.M.Buchstaber and
S.P.Novikov Eds.), Amer. Math. Soc. Transl. (2), Vol. {\bf 179}, AMS, 
Providence, RI, 1997, pp. 173--189.

\bibitem[Zo5]{Zorich5} 
\bysame, \emph{How Do the Leaves of a Closed $1$-form Wind around a Surface?}, in
{\it Pseudoperiodic Topology }(V.I.Arnol'd, M.~Kontsevich and A.~Zorich Eds.), 
Amer. Math. Soc. Transl. (2), Vol. \textbf{197}, AMS, Providence, RI, 1999, 
135--178.

\bibitem[Zo6]{Zorich6} \bysame, \emph{Flat surfaces},
 Frontiers in Number Theory, Physics, and Geometry I, 2006, 437--583.

\end{thebibliography}
\end{document}